\documentclass[12pt]{article}

\usepackage{amsmath,amsthm,amsfonts}

\parskip=\medskipamount
\def\vectorfields#1{{\cal X}(#1)}
\def\ov#1{\overline{#1}}
\def\fpd#1#2{\frac{\partial #1}{\partial #2}}
\def\R{{\rm I\kern-.20em R}}
\def\sode{second-order differential equation field}
\newcommand{\lie}[1]{\mathcal{L}_{#1}}

\newcommand{\mybox}[1]{\makebox(0,0){\footnotesize{#1}}}

\def\Sec{\mathop{\mathrm{Sec}}}
\def\la{{\mathfrak g}}
\def\pM{{\pi}^{\scriptscriptstyle M}}
\def\pTM{\pi^{\scriptscriptstyle TM}}
\def\pTTM{\pi^{\scriptscriptstyle TTM}}
\def\otau{{\ov\tau}}
\def\psiM{\psi^{\scriptscriptstyle M}}
\def\psiTM{\psi^{\scriptscriptstyle TM}}
\def\psiTTM{\psi^{\scriptscriptstyle TTM}}
\def\og{{\ov \la}}

\def\Ad{\mathop{\mathrm{ad}}\nolimits}
\def\ad{\Ad}
\def\Ker{\mathop{\mathrm{ker}}}
\def\Im{\mathop{\mathrm{im}}}
\def\id{\mathrm{id}}

\def\vf#1{{\displaystyle\frac{\partial}{\partial #1}}}

\def\clift#1{#1^{\scriptscriptstyle\mathrm{C}}}

\def\vlift#1{#1^{\scriptscriptstyle\mathrm{V}}}
\def\cgamma{\clift{\gamma}}
\def\comega{\clift{\omega}}
\def\vgamma{\vlift{\gamma}}
\def\vomega{\vlift{\omega}}

\def\lequiv{\lbrack\!\lbrack}
\def\requiv{\rbrack\!\rbrack}
\def\equivcl#1{\lequiv#1\requiv}

\def\D{\mathcal{D}}
\def\X{\mathcal{X}}

\def\conn#1#2#3{\setbox1=\hbox{$\scriptstyle{#2}{#3}$}%
\setbox2=\hbox to\wd1{$\hfil\scriptstyle{#1}\hfil$}
\Gamma^{\!\box2}_{\!\box1}}
\def\barconn#1#2#3{\setbox1=\hbox{$\scriptstyle{#2}{#3}$}%
\setbox2=\hbox to\wd1{$\hfil\scriptstyle{#1}\hfil$}
\ov{\Gamma}^{\!\box2}_{\!\box1}}
\def\adjconn#1#2#3{\setbox1=\hbox{$\scriptstyle{#2}{#3}$}%
\setbox2=\hbox to\wd1{$\hfil\scriptstyle{#1}\hfil$}
\Upsilon^{\!\box2}_{\!\box1}}

\def\onehalf{{\textstyle\frac12}}

\def\lie#1{\mathcal{L}_{#1}}

\newtheorem{thm}{\bf Theorem}
\newtheorem{prop}[thm]{\bf Proposition}

\begin{document}

\title{Reduction and reconstruction aspects of
second-order dynamical systems with symmetry}
\author{M.\ Crampin${}^{a}$\footnote{email: crampin@btinternet.com}\, and\, T.\ Mestdag${}^{a,b}$\footnote{email: tom.mestdag@ugent.be}\\
{\small ${}^a$Department of Mathematical Physics and Astronomy }\\
{\small Ghent University, Krijgslaan 281, B-9000 Ghent, Belgium}\\
{\small ${}^b$ Department of Mathematics}\\
{\small University of Michigan, 530 Church Street, Ann Arbor, MI
48109, USA}}
%\\
%{\small email: Tom.Mestdag@UGent.be}}
\date{}

\maketitle

{\small {\bf Abstract.} We examine the reduction process of a system
of second-order ordinary differential equations which is invariant
under a Lie group action. With the aid of connection theory, we
explain why the associated vector field decomposes in three parts
and we show how the integral curves of the original system can be
reconstructed from the reduced dynamics. An illustrative example
confirms the results.
\\[2mm]
\\[2mm]
}

\section{Introduction}

This paper is concerned with second-order dynamical systems, or in
other words systems of second-order ordinary differential equations,
which admit a Lie group of symmetries; the question it deals with is
how the symmetry group can be used to simplify the system (reduction),
and how, knowing a solution of the simplified system one can find a
solution of the original system (reconstruction).

This is not of course a new problem:\ reduction and reconstruction
have been studied in a number of different contexts in geometry and
dynamics. We mention in particular the following topics in which the
object of interest is a second-order system:
\begin{itemize}
\item the geodesics of a manifold with a Kaluza-Klein metric, and the
Wong equations \cite{Mont};
\item Lagrange-Poincar\'e equations and reduction by stages
\cite{Cendra,Me2};
\item non-Abelian Routh reduction \cite{Marsden};
\item Chaplygin systems \cite{Bloch,Cortes}.
\end{itemize}
We mention these studies in order to emphasise the fact that we are
engaged in this paper in something different from any of them.  Each
of the listed studies deals with a special class of second-order
system --- for example, with systems of Euler-Lagrange equations,
that is, equations derived by variational methods from a Lagrangian.
We, by contrast, deal with systems of second-order equations pure
and simple:\ we make no assumptions about how they are derived, and
make no appeal to properties other than the property of being
second-order and being invariant under a suitable symmetry group.
(For the sake of clarity, we should perhaps remind the reader that
by no means all second-order systems are of Euler-Lagrange type.)
This has two important consequences.  First, our results are more
general than any of those obtained above:\ indeed, in some sense
they must subsume the main features of the results of any of these
particular studies.  We shall show in fact how a system with
symmetries may be reduced to a coupled pair of sets of equations,
one of second order and one of first order. Second, our methods must
be likewise more general:\ one cannot use variational methods, for
example, if one is not dealing with a system of Euler-Lagrange
equations.  Of course one cannot say much without invoking some type
of auxiliary machinery; but we use the minimum possible, just a
connection.

So far as we are aware, the only previous attempt at anything like
this bare hands approach is \cite{Bullo}, which deals with control
systems of mechanical type, but with only a one-dimensional symmetry
group; we deal with non-Abelian groups of arbitrary dimension.

When we come to discuss the second-order case we shall represent a
system of second-order ordinary differential equations by a special
kind of vector field.  There is therefore a lot to be said for
starting off by considering the question for arbitrary vector fields,
or in other words for first-order dynamical systems.

Before proceeding, we must again make it clear exactly what we are
attempting to do here.  Perhaps the most studied example of an
invariant first-order system is that of a Hamiltonian system with a
symmetry group.  The symmetry of such a system provides an
equivariant moment map that is invariant under the Hamiltonian
dynamics, and this feature plays an important role in the reduction
process (this is discussed for example in \cite{MW}; a Lagrangian
version is given in \cite{MMR}, together with an account of the
reconstruction of the solutions in both the Lagrangian and
Hamiltonian context).  By contrast, we deal with vector fields pure
and simple; since we have no Hamiltonian structure at our disposal,
we cannot appeal to properties of momentum maps and so on.

Suppose given a dynamical system, represented by a vector field $X$
on a manifold $M$, which admits a Lie group $G$ of symmetries.
Suppose further that $G$ acts freely and properly on $M$ so that $M$
is a principal bundle over a manifold $B$ with group $G$; let
$\pM:M\to B$ be the projection. Since $X$ is invariant under the
action of $G$ there is a vector field $\ov{X}$ on $B$ which is
$\pM$-related to $X$; this is the reduced dynamical system.

As a process of reduction, however, this is clearly incomplete in the
sense that there is no way of reconstructing the original dynamical
system from the reduced one; any two $G$-invariant dynamical systems
on $M$ which differ by a $\pi^M$-vertical vector field (which is
necessarily also $G$-invariant) have the same reduced dynamics. To
see what is at stake, let us introduce coordinates $(x^i,x^a)$ on $M$
such that the orbits of $G$, or in other words the fibres of $M\to
B$, are given by $x^i=\mbox{constant}$; the $x^i$ may therefore be regarded as
coordinates on $B$. Let us denote by $\tilde{E}_a$ a basis for the
fundamental vector fields on $M$ generated by the $G$-action; then
\[
\tilde{E}_a=K_a^b\vf{x^b}
\]
for some non-singular matrix-valued function $(K^b_a)$. Suppose
further that we have at our disposal a distribution on
$M$ which is transverse to the fibres and $G$-invariant. Such a
distribution will be spanned by vector fields $X_i$ of the form
\[
X_i=\vf{x^i}-\Lambda_i^a\tilde{E}_a
\]
for certain functions $\Lambda_i^a$. We may then write
\[
X=Y^iX_i+Z^a\tilde{E}_a=Y^i\vf{x^i}+(Z^a-\Lambda^a_iY^i)\tilde{E}_a.
\]
The necessary and sufficient conditions for $X$ to be $G$-invariant
are that $Y^iX_i$ and $Z^a\tilde{E}_a$ are separately $G$-invariant.
In particular, the $Y^i$ are independent of the $x^a$, so that
$Y^i\partial/\partial x^i$ may be regarded as a vector field on the
base manifold $B$:\ this is the reduced dynamical system $\ov{X}$,
of course. The integral curves of $X$ are solutions of the
differential equations
\[
\dot{x}^i=Y^i,\qquad \dot{x}^a=(Z^b-\Lambda^b_i\dot{x}^i)K^a_b.
\]
The equations of the first set define the integral curves of the
reduced dynamical system.  The remainder can in principle be used to
reconstruct an integral curve of the original dynamical system from a
known integral curve of the reduced one.

This description of the process is somewhat disingenuous:\ a
fibre-transverse $G$-invariant distribution on a principal
$G$-bundle is of course just a connection, or more accurately a
principal connection.  This observation gives us the opportunity to
describe the reduction in a coordinate-independent way, as is
clearly desirable; when we do so, moreover, the reconstruction step
acquires a more transparent geometrical interpretation than is
apparent from the description above.  Our basic contention is that
the simplest additional machinery that is required to give a
geometrically coherent account of the reduction and reconstruction
of dynamical systems with symmetry is a connection; and we aim to
show how these processes work in that context for second-order
dynamical systems.

A second-order dynamical system can be represented by a vector field
$\Gamma$ on the tangent bundle $TM$ of a differentiable manifold $M$,
of the form
\[
\Gamma=v^\alpha\vf{x^\alpha}+\Gamma^\alpha\vf{v^\alpha},
\]
where the $v^\alpha$ are the fibre coordinates. Given a vector field
$X$ on $M$, let us denote by $\clift{X}$ its complete, or tangent,
lift to $TM$ and by $\vlift{X}$ its vertical lift. Then in terms of
the structure described above in the first-order case, we may express
a second-order differential equation field $\Gamma$ on $TP$ in the form
\[
\Gamma=v^i\clift{X}_i+v^a\clift{\tilde{E}}_a+D^i\vlift{X}_i
+D^a\vlift{\tilde{E}}_a.
\]
It turns out that in order for $\Gamma$ to be invariant under the
action of $G$ on $TM$ induced from its action on $M$, each of the
three components $v^a\clift{\tilde{E}}_a$, $D^a\vlift{\tilde{E}}_a$
and $v^i\clift{X}_i+D^i\vlift{X}_i$ must be invariant. The last of
these represents a second-order dynamical system, albeit in a
generalized sense; the last two define the coupled first- and
second-order equations which constitute the reduced system mentioned
earlier. It is our aim to explain how this decomposition arises,
with the help of connection theory; and to discuss the processes of
reduction and reconstruction from this standpoint.

In the following section we discuss the first-order case in greater
detail.  In Section 3 we deal with the connection theory required for
the analysis of the reduction and reconstruction of second-order
systems, and in Section 4 we carry out that analysis. Section 5 is
devoted to consideration of an example.

It will become apparent that a particular kind of Lie algebroid, the
Atiyah algebroid of a principal bundle, plays an important role in
the theory.  In fact one can locate the case discussed here in a
more general framework consisting of Lie algebroids and anchored
vector bundles.  Investigation of this aspect of the matter
continues.

\section{First-order systems}

As before, we suppose that $M$ is a manifold on which a Lie group $G$
acts freely and properly to the right; we denote the action by $\psiM
: G \times M \to M$.  Then $\pM: M \to M/G=B$ is a principal fibre
bundle and $\pM\circ \psiM_g =\pM$ for all $g\in G$.  For $\xi\in\la$,
the Lie algebra of $G$, we denote by $\xi_M$ the fundamental vector
field corresponding to $\xi\in\la$, that is, the infinitesimal
generator of the 1-parameter group $\psi^M_{\exp(t\xi)}$ of
transformations of $M$.

The $G$-action on $M$ can be extended to a $G$-action $\psiTM: G
\times TM \to TM$ on the tangent manifold $\tau:TM\to M$, given by
$(g,v_m) \mapsto T_m \psiM_g (v_m)$, for $m\in M$, $v_m\in T_mM$.
This action equips $TM$ with the structure of a principal fibre bundle
over $TM/G$, with projection $\pTM$.  Then, of course, $\pTM\circ
\psiTM_g=\pTM$ and $T\pM\circ \psiTM_g=T\pM$, where $T\pM: TM \to
T(M/G)$; on the other hand, $\tau\circ\psiTM_g=\psiM_g$.  We also have
an action $\psi^{\mathcal{X}}$ on the space of sections of $TM\to M$,
that is, on $\vectorfields{M}$, the space of vector fields on $M$,
given by
\[
\psi^{\mathcal{X}}(X)(m)=\psiTM_g(X(\psiM_{g^{-1}}(m))).
\]
A vector field $X$ on $M$ is $G$-invariant if for all $g\in G$
\[
X(\psiM_g(m))=\psiTM_g(X(m)), \quad\mbox{or equivalently}\quad
\psi^{\mathcal{X}}_g(X)=X.\]
If $X$ is $G$-invariant then $[\xi_M,X]=0$ for all $\xi\in\la$. If
$G$ is connected, as we shall generally assume to be the case, this
is a sufficient as well as a necessary condition for invariance.

The fundamental vector fields satisfy
$\psi^{\mathcal{X}}_g(\xi_M)=(\ad_{g^{-1}}\xi)_M$, where $\Ad$ is the
adjoint action of $G$ on $\la$.

For all $m\in M$, $\pTM_m$ induces an isomorphism $T_mM \to
(TM/G)_{\pM(m)}$, the fibre of $TM/G$ over $\pM(m)\in M/G$, and thus
also an isomorphism $TM \to (\pM)^*TM/G$.  As a consequence of this
property there is a 1-1 correspondence between invariant vector
fields on $M$ and sections of the vector bundle $\otau: TM/G \to
M/G$ (see e.g.\ \cite{Me}).  The vector bundle $\otau$ has the
structure of a Lie algebroid:\ the anchor map $\varrho: TM/G \to
T(M/G)$ is given by $\equivcl{v} \mapsto T\pM(v)$ (here and below
$\equivcl{\cdot}$ represents the equivalence class of the argument
under $G$-equivalence, or in other words its $G$-orbit), which is
independent of the choice of $v\in\equivcl{v}$ because of the
property $T\pM \circ \psiTM_g =T\pM$; the bracket of two sections of
$TM/G$ is given by the bracket of the associated invariant vector
fields.  With this Lie algebroid structure $\otau$ is called the
Atiyah algebroid of the principal $G$-bundle $\pi^M$ \cite{Mac}.

The fibre-linear map $T\pM:TM\to T(M/G)$ is surjective on the fibres.
The kernel of the induced map $TM\to(\pM)^*T(M/G)$ is isomorphic to
the bundle $M\times \la \to M$; the identification of $M\times \la$ as
a subbundle of $TM$ is given by $(m,\xi) \mapsto \xi_M (m)$.

A connection on $\pM$ is a right splitting $\gamma$ of the short
exact sequence
\begin{equation} \label{short1}
0 \to M\times \la \to TM \stackrel{T\pM}{\to} (\pM)^*T(M/G) \to 0
\end{equation}
of vector bundles over $M$.  The corresponding left splitting $TM \to
M\times\la$ will be denoted by $\omega$.  We will write $\varpi$ for
its projection on $\la$.  The distinction between $\omega$ and
$\varpi$ can be made clear as follows.  If we identify $M\times \la$
with a subbundle of $TM$ then $\omega$ may be thought of as a type
$(1,1)$ tensor field on $M$; we have $\omega(\xi_M) = \xi_M$,
while $\varpi(\xi_M)=\xi$.  Needless to say, both $\omega$ and
$\varpi$ vanish on $\Im(\gamma)$.  The map $\gamma$ may be thought of
as the horizontal lift, the $\la$-valued 1-form $\varpi$ as the
connection form.

If $\varpi$ satisfies $\varpi(\psiTM_g v)= \Ad_{g^{-1}}\varpi(v)$ the
connection is said to be principal.  Equivalently, principal
connections are right splittings $\gamma$ with the property
$\gamma(\psiM_g m, \ov v) = \psiTM_g\gamma(m,\ov v)$ for all
$\ov{v}\in T_{\pM(m)}(M/G)$.  The condition for the connection to be
principal when expressed in terms of $\omega$ is simply that it is
invariant under the $G$-action on $TM$, that is, that
$\omega\circ\psiTM_g=\psiTM_g\circ\omega$.

The manifold $M\times\la$ comes equipped with the right action
$g\mapsto(\psiM_g,\Ad_{g^{-1}})$; we denote by $\og = (M\times \la)/G$
its quotient under this action.  We remark that $\og$ is the (total
space of) the vector bundle associated with the principal $G$-bundle
$\pM$ by the adjoint action of $G$ on $\la$; it is often called the
adjoint bundle.  When we take the quotient of the exact sequence
(\ref{short1}) under the action of $G$ we obtain the following short
exact sequence of vector bundles over $M/G$:
\begin{equation}\label{short2}
0 \to \og \to TM/G \stackrel{\varrho}{\to} T(M/G) \to 0,
\end{equation}
which is called the Atiyah sequence \cite{Mac}.  If $\gamma$ is a
principal connection on $\pM$ then $\pTM(\gamma(m,\ov v))$ is
independent of the choice of $m\in\equivcl{m}=\otau(\ov v)$ because of
the invariance of $\gamma$; if we set $\ov \gamma(\ov v) =
\pTM(\gamma(m,\ov v))$ then $\ov{\gamma}:T(M/G)\to TM/G$ is
well-defined and satisfies $\varrho\circ\ov\gamma=\id$, and is
therefore a right splitting of the Atiyah sequence.  This establishes
a correspondence between principal connections on $\pM$ and splittings
of the Atiyah sequence, which is actually 1-1.  If $\ov\gamma$ is a
right splitting of the Atiyah sequence, the corresponding left
splitting will be denoted by $\ov\omega$.

If $\varpi$ is the connection form of a principal connection on
$\pM$ and $X$ a $G$-invariant vector field on $M$ then $\varpi(X)$
is a $\la$-valued function on $M$ which satisfies
$\varpi(X)\circ\psiM_g=\ad_{g^{-1}}\varpi(X)$. So the map
$m\mapsto(m,\varpi(X)(m))\in M\times\la$ is constant on the orbits
of the $G$-action, and therefore defines a section of $\ov{\la}\to
M/G$.

We next describe the reduction of a $G$-invariant vector field.  As we
pointed out earlier, a $G$-invariant vector field $X$ can be
identified with a section $\tilde X$ of $TM/G$, given by $\tilde X
(\pM(m))= \pTM(X(m))$.  If $\gamma$ is a principal connection then
$\tilde X$ can in turn be decomposed into a vector field $\ov X =
\varrho \tilde X$ on $M/G$ and a section $\ov\omega(\tilde X)$ of
$\og$.  The relation between $X$ and $\ov X$ is $\ov
X(\pM(m))=T\pM(X(m))$; $\ov{X}$ is the reduced vector field of $X$.

The decomposition may be described in a slightly different way.  Given
a connection $\gamma$, any $X\in\vectorfields M$ can be decomposed
into its horizontal and vertical components with respect to $\gamma$;
the horizontal component is determined by a section of
$(\pM)^*T(M/G)$, and the vertical component can be identified with
$\omega(X)$ and hence with the $\la$-valued function $\varpi(X)$.
When $X$ is $G$-invariant and $\gamma$ is a principal connection, the
horizontal component is the horizontal lift of a vector field on
$M/G$, namely $\ov{X}$; and the section of $\og$ that $\varpi(X)$
defines is just $\ov\omega(\tilde X)$.

The following fact about the integral curves of an invariant vector
field is well-known. Suppose
that $t\mapsto c(t)$ is an integral curve of $X$, so that $\dot c =
X\circ c$. Then the curve $t\mapsto \ov c(t) = \pM(c(t))$ in $M/G$
is an integral curve of $\ov X$, that is,
\begin{equation}\label{intcurveovX}
\dot{\ov c} = \ov X\circ\ov{c}.
\end{equation}
Indeed, $\dot{\ov c} = T\pM\circ\dot c = T\pM \circ(X\circ c)=
\ov{X}\circ(\pM\circ c)$.
In fact an integral curve $c$ of $X$ is completely determined by the
underlying integral curve $\ov c$ of $\ov X$ and a curve $t\mapsto g(t)$
in $G$. To see this, note that there is a unique curve ${\ov c}^\gamma$ in
$M$, the horizontal lift of $\ov{c}$ through $c(0)$, such that
\begin{itemize}
\item ${\ov c}^\gamma$ projects onto $\ov c$
\item ${\ov c}^\gamma(0)=c(0)$
\item the tangent $\dot{{\ov c}}^\gamma$ to ${\ov c}^\gamma$ is
everywhere horizontal (so that ${\ov c}^\gamma$ satisfies
$\dot{{\ov c}}^\gamma=\gamma({\ov c}^\gamma,\dot{\ov c})$).
\end{itemize}
Then since $c$ also projects onto $\ov{c}$ there is a curve $t\mapsto
g(t)\in G$, with $g(0)=e$ (the identity element of $G$), such that
$c(t)=\psiM_{g(t)}{\ov c}^\gamma(t)$.  Now let $\theta$ be the
Maurer-Cartan form of $G$:\ then $t\mapsto \theta(\dot{g}(t))$ is a
curve in $\la$.  By differentiating the equation $c=\psiM_{g}{\ov
c}^\gamma$ we see that $\dot{g}$ must satisfy
\begin{equation}\label{intcurveX}
\dot c = \psiTM_{g} \left((\theta(\dot{g}))_M\circ{\ov c}^\gamma
+ \dot{{\ov c}}^\gamma\right).
\end{equation}
But $\dot c=X\circ c=X\circ (\psiM_{g}{\ov c}^\gamma)=
\psiTM_{g}(X\circ \ov{c}^\gamma)$, whence
\[
(\theta(\dot{g}))_M\circ{\ov c}^\gamma + \dot{{\ov c}}^\gamma=
X\circ{\ov c}^\gamma.
\]
The first term on the left-hand side (which when evaluated at $t$ is
the value at ${\ov c}^\gamma(t)$ of the fundamental vector field
corresponding to $\theta(\dot{g}(t))\in\la$) is vertical, the second
horizontal, so this equation is simply the decomposition of $X$ into
its horizontal and vertical components, at any point of $\ov{c}^\gamma$.
In particular,
\begin{equation} \label{geq}
\theta(\dot{g}) = \varpi(X\circ{\ov c}^\gamma),
\end{equation}
the right-hand side being of course a curve in $\la$.  This is a
differential equation for the curve $g$, and has a unique solution
with specified initial value.  (When $G$ is a matrix group
$\theta(\dot{g})=\dot{g}g^{-1}$; and the equation
$\theta(\dot{g})=\xi$, where $t\mapsto\xi(t)$ is a curve in $\la$,
can be written $\dot{g}=\xi g$, from which the assertion is obvious.
See for example \cite{Sharpe} for the general case.)  Thus the curve
$g$ is uniquely determined by equation~(\ref{geq}) and the initial
condition $g(0)=e$.

We can conclude the following.
\begin{prop}
Given a principal connection $\gamma$, one can reconstruct the
integral curves of the $G$-invariant vector field $X$ from those of
the reduced vector field $\ov{X}$. In order to carry out the
reconstruction one needs to solve successively
\[
\left\{ \begin{array}{lllll} \dot{\ov c} & = & \ov X (\ov c)&&
\mbox{for $\ov c$} \\
\dot{{\ov c}}^\gamma &= & \gamma({\ov c}^\gamma,\dot{\ov c})
&& \mbox{for ${\ov c}^\gamma$} \\
\theta(\dot{g}) &=& \varpi(X\circ{\ov c}^\gamma) && \mbox{for $g$},
\end{array}
\right.
\]
to obtain finally the integral curve $c=\psiM_g{\ov c}^\gamma$ of
$X$.\end{prop}

We now give some explicit expressions for the decomposition.  We
consider first the construction of a basis of vertical vector
fields, that is, vector fields tangent to the orbits of the
$G$-action.  There are in fact two possible choices, at least
locally, corresponding to what are sometimes called, as in
\cite{Bloch}, the `moving basis' and the `body-fixed basis'. The
reference is to rigid body dynamics; the point is that the
body-fixed basis is invariant.

Let $\{E_a\}$ be a basis for $\la$, and $C^c_{ab}$ the corresponding
structure constants.  The moving basis consists of the fundamental
vector fields $(E_a)_M$.  These vector fields are not of course
invariant:\ in fact for any fundamental vector field $\xi_M$,
$\psi^{\mathcal{X}}_g(\xi_M)=(\Ad_{g^{-1}}\xi)_M$, as we pointed out
earlier.  We will usually write
$\tilde{E}_a$ instead of $(E_a)_M$ for convenience.

The definition of the body-fixed basis depends on a choice of local
trivialization of $\pM: M \to M/G$.  Let $U\subset M/G$ be an open
set over which $M$ is locally trivial.  The projection $\pM$ is
locally given by projection onto the first factor in $U\times G \to
U$, and the action by $\psiM_g(x,h)=(x,hg)$. The maps
\[
{\ov E}_a:U\to (M\times\la)/G|_U\quad
\mbox{by}\quad x \mapsto \equivcl{(x,e),E_a}
\]
will give a local basis for $\Sec(\og)=\Sec((M\times\la)/G)$ over $U$.
These maps can be considered as sections of $TM/G\to M/G$ by means of
the identification
\[
{\ov E}_a \in \Sec(\og) \quad \Longleftrightarrow \quad {\ov E}_a:
x\mapsto \pM\big( \tilde{E}_a (x,e) \big) \in \Sec(TM/G).
\]
Recall first that the injection $M\times \la \to TM$ is given by
$(m,\xi) \mapsto \xi_M (m)$. Further, it is clear that the two
elements $((x,e),\xi)$ and $((x,g), \Ad_{g^{-1}}\xi)$ of $M\times\la$ belong
to the same equivalence class in $\og$. This is in perfect agreement
with the above identification, since
\[
\pTM \big( \xi_M (x,e) \big) = \pTM \big( (\Ad_{g^{-1}}\xi)_M (x,g) \big).
\]
Now sections of $TM/G$ can be lifted to invariant vector fields on $M$.
For the above sections, the invariant vector fields are
\[
{\hat E}_a: (x,g) \mapsto (\Ad_{g^{-1}} E_a)_M(x,g)
= \psiTM_g \big((E_a)_M (x,e)\big).
\]
Then, indeed, $\psi^{\mathcal{X}}_g({\hat E}_a)={\hat E}_a$.  The
corresponding basis for $\Sec(M\times\la)$ is given by the
sections $(x,g) \mapsto ((x,g), \Ad_{g^{-1}} E_a)$.  By contrast the
fundamental vector fields ${\tilde E}_a$, identified as sections of
$M\times\la\to M$, are given by $(x,g)\mapsto ((x,g), E_a)$.  The
relation between the two sets of vector fields can be expressed as
${\hat E}_a(x,g)= A_a^b(g){\tilde E}_b(x,g)$ where $(A_a^b(g))$ is the
matrix representing $\Ad_{g^{-1}}$ with respect to the basis $\{E_a\}$ of
$\la$.

In fact the body-fixed basis $\{\hat{E}_a\}$ associated with a local
trivialization $(\pM)^{-1}U\equiv U\times G$ is obtained just by
transferring to $(\pM)^{-1}U$ the right-invariant vector fields on $G$
associated with the basis $\{E_a\}$ of $\la$.  The moving basis, on
the other hand, corresponds to the left-invariant vector fields on $G$
associated with the basis $\{E_a\}$.

Let us take coordinates $(x^i,x^a)$ on $M$ such that $(x^i)$ are
coordinates on $U$, $(x^a)$ coordinates on the fibre.  Then
there are `action functions' (so-called in \cite{Bloch}) such
that ${\tilde E}_a=K_a^b(x^c)\partial/\partial x^b$.
The relation $[{\tilde E}_a, {\tilde E}_b] = C^c_{ab} {\tilde E}_c$
leads to the property
\[
K^c_a \fpd{K^d_b}{x^c} - K^c_b \fpd{K^d_a}{x^c} = C^e_{ab}K^d_e.
\]
The invariance of the vector fields ${\hat E}_a$ can be expressed as
\[
[{\tilde E}_b, {\hat E}_a] = 0 \quad \Longleftrightarrow
\quad {\tilde E}_b(A^c_a)+A^d_aC^c_{bd}=0.
\]
We can use these differential equations as another way of constructing
a body-fixed basis, as follows.  We seek local vector fields
$\{\hat{E}_a\}$, given in terms of the moving basis $\{\tilde{E}_a\}$
by $\hat{E}_b=A_b^a\tilde{E}_a$ where $(A_b^a)$ is a locally defined
non-singular matrix-valued function on $M$, which are $G$-invariant,
which is to say that $[\tilde{E}_a,\hat{E}_b]=0$ for all $a$ and $b$.
Thus, as above, the $A^a_b$ must satisfy
\begin{equation}\label{eqnforA}
\tilde{E}_a(A_b^c)+C^c_{ad}A_b^d=0.
\end{equation}
This is a system of linear partial differential equations for the unknowns
$A_b^a$. The integrability conditions
\[
[\tilde{E}_a,\tilde{E}_b](A_c^d)+C^d_{be}\tilde{E}_a(A_c^e)
-C^d_{ae}\tilde{E}_b(A_c^e)=0
\]
are identically satisfied by virtue of the Jacobi identity.  The
equations therefore have solutions locally on $M$, and a solution can
be specified by choosing a local cross-section of the $G$ action and
specifying the value of $(A_b^a)$ on it; the natural choice, which we
make, is to take it to be the identity matrix. The $A^a_b$ will then
be independent of the $x^i$.

A simple calculation shows that
\[
[\hat{E}_a,\hat{E}_b]=-A^d_aA^e_b\bar{A}^c_fC^f_{de}\hat{E}_c,
\]
where the $\bar{A}^a_b$ are the components of the matrix inverse to
$(A^a_b)$.  On the other hand, if we write
$[\hat{E}_a,\hat{E}_b]=-\hat{C}^c_{ab}\hat{E}_c$ then the coefficients
$\hat{C}^c_{ab}$ must be $G$-invariant, since everything else in the
equation is.  It follows that the value of $\hat{C}^c_{ab}$ along any
fibre of $(\pM)^{-1}U\to U$ is the same as its value on the section
which determines the local trivialization, that is, where $g=e$; if we
take $(A^a_b)$ to be the identity there we obtain
$\hat{C}^c_{ab}=C^c_{ab}$, that is,
$[\hat{E}_a,\hat{E}_b]=-C^c_{ab}\hat{E}_c$ (as one would expect).

We now consider the horizontal vector fields. We have at our disposal
the local coordinate basis $\{\partial/\partial x^i\}$ of
$\vectorfields{T(M/G)}$; we put
\[
{\ov X}_i=\ov\gamma\left(\fpd{}{x^i}\right)\in\Sec(TM/G).
\]
The sections $\{{\ov X}_i,{\ov E}_a\}$ form
a basis of $\Sec(TM/G)$. They can be lifted to a basis $\{{X}_i,
{\hat E}_a\}$ of $\vectorfields{M}$, consisting only of invariant
sections. Then
\[
{X}_i (x,g) = \gamma \left((x,g),\fpd{}{x^i}\biggr|_{x}\right).
\]
If we set
\[
{X}_i =  \fpd{}{x^i} - \gamma_i^b(x^i,x^a) {\hat E}_b=
\fpd{}{x^i} - \gamma_i^b(x^i,x^a) A_b^c(x^a) {\tilde E}_c,
\]
then invariance of $X_i$ amounts to
\[
[{\tilde E}_b, {X}_i] = 0 \quad \Longleftrightarrow  \quad {\tilde E}_b
(\gamma_i^c) =0 \quad \Longleftrightarrow  \quad \fpd{\gamma_i^c}{x^b}
=0.
\]
For future use we calculate $[\hat{E}_a,X_i]$ here also. We have
\[
[\hat{E}_a,X_i]=-X_i(A_a^c)\bar{A}_c^b\hat{E}_b.
\]
Now $\partial A_a^b/\partial x^i=0$, and so
\begin{equation}\label{XofA}
X_i(A_a^c)=-\gamma_i^dA_d^e\tilde{E}_e(A_a^c)=\gamma_i^dC^c_{ef}A_d^eA^f_a
=\gamma_i^dC^e_{da}A^c_e,
\end{equation}
so that
\begin{equation}\label{hatEXbrac}
[\hat{E}_a,X_i]=\gamma_i^bC^c_{ab}\hat{E}_c,
\end{equation}
assuming as we may that $[\hat{E}_a,\hat{E}_b]=-C^c_{ab}\hat{E}_c$.

A vector field $X$ on $M$ can be written as
$X= Y^j {X}_j + Y^b {\hat E}_b$. If $X$ is invariant then
\[
\fpd{Y^j}{x^c} =0 \quad \mbox{and} \quad \fpd{Y^b}{x^c} =0.
\]
If alternatively we set $X= Y^j {X}_j + Z^c{\tilde E}_c$, where $Z^c
=  A^c_bY^b$, then the second invariance condition becomes
\begin{equation}\label{Zinv}
{\tilde E}_d(Z^c)+C_{de}^c Z^e=0.
\end{equation}
An invariant vector field $X$ projects onto the section ${\tilde
X}:(x^i) \mapsto Y^j(x^i)X_j + Y^a(x^i){\ov E}_a$ of $TM/G$ and the
vector field $\ov{X}:(x^i) \mapsto Y^j(x^i)\partial/\partial x^i$ on
$M/G$.  Finally, we have
\begin{itemize}
\item $\omega(X)=Y^a{\hat E}_a=Z^a\tilde{E}_a \in TM$;
\item ${\ov\omega}(\tilde X)=Y^a{\ov E}_a \in \Sec(\og)$;
\item $\varpi(X)=Z^aE_a \in C^\infty(M,\la)$.
\end{itemize}
In fact, as we pointed out before, when $X$ is invariant $\varpi(X)$
defines a section of $\ov{\la}\to M/G$.  We now wish to explain how
one can recognise a section of $\ov{\la}$ in terms of coordinates.
Recall that a section of $\ov{\la}$ can be thought of as a function
$M\to\la$ which is constant on the equivalence classes of the
equivalence relation defining the associated bundle structure; that
is, a $\la$-valued function $s$ on $M$ such that
$s\circ\psiM_g=\ad_{g^{-1}}s$.  Assuming as always that $G$ is
connected, we may equivalently write this condition as
$\xi_{M}(s)+[\xi,s]=0$ for any $\xi\in\la$, where the bracket is the
Lie algebra bracket of $\la$.  We may express $s$ as $s=s^aE_a$ with
respect to a basis $\{E_a\}$ of $\la$; in terms of the components
$s^a$ of $s$ the condition for $s$ to define a section is
\begin{equation}\label{secassoc}
\tilde{E}_b(s^a)+C^a_{bc}s^c=0.
\end{equation}
This makes clear the significance of equation~(\ref{Zinv}).

\section{Second-order diagrams and connections}

In this section we discuss the connection theory relevant to
second-order dynamical systems. Before we do so, however, it will be
convenient to make some remarks about splittings of short exact
sequences in general; these remarks will be useful later.

If $ 0\to \Ker f \to A \stackrel{f}{\to} B\to 0 $ and $ 0\to \Ker g
\to B \stackrel{g}{\to} C\to 0$ are two short exact sequences of
vector bundles over the same manifold, then the sequence
\[
0\to \Ker (g\circ f) \to A \stackrel{g\circ f}{\to} C \to 0
\]
is also exact. Moreover the restriction of $f$ to $\Ker(g\circ f)$
gives rise to a fourth short exact sequence
\[
0\to \Ker f \to \Ker(g\circ f) \stackrel{f}{\to} \Ker g \to 0.
\]
In summary, we can draw the following commutative diagram:

\setlength{\unitlength}{1cm}
\begin{picture}(14.29,6.5)(-3.5,4)
\put(5.1,9.9){\vector(1,0){2.0}} \put(1.14,7.3){\vector(1,0){2.0}}
\put(5.1,7.3){\vector(1,0){2.0}} \put(1.14,4.5){\vector(1,0){2.0}}
\put(5.1,4.5){\vector(1,0){2.0}} \put(1.14,9.9){\vector(1,0){2.0}}

\put(8.13,9.5){\vector(0,-1){1.7}}
\put(4.2,6.8){\vector(0,-1){1.7}}
\put(4.2,9.5){\vector(0,-1){1.7}}
\put(0.6,6.8){\vector(0,-1){1.7}}
\put(0.6,9.5){\vector(0,-1){1.7}}
\put(8.13,6.8){\vector(0,-1){1.7}}

\put(0.6,9.9){\mybox{\fbox{$\Ker f$}}}
\put(4.2,9.9){\mybox{$\Ker(g\circ f)$}}
\put(8.13,9.9){\mybox{\fbox{$\Ker g$}}}

\put(0.6,7.3){\mybox{$\Ker f$}} \put(4.2,7.3){\mybox{\fbox{$A$}}}
\put(8.13,7.3){\mybox{$B$}}

\put(0.6,4.5){\mybox{$\fbox{0}$}} \put(4.2,4.5){\mybox{$C$}}
\put(8.13,4.5){\mybox{\fbox{$C$}}}

\end{picture}

The following facts are immediate. Suppose given splittings
$\gamma_1:B\to A$ and $\gamma_2:C\to A$. If we set
$\gamma_3=f\circ\gamma_2: C \to B$ then $\gamma_3$ is also a
splitting.  If, in addition, $\gamma_2(C)\subset\gamma_1(B)$ then
$\gamma_2=\gamma_1\circ\gamma_3$. Furthermore, $\gamma_1$ restricts
to a splitting $\Ker g \to \Ker(g\circ f)$.  Therefore, given
$\gamma_1$ and $\gamma_2$, each element of $A$ can be uniquely
decomposed into three parts, one in $C$, one in $\Ker f$ and one in
$\Ker g$.  We will use this observation for our decomposition of
\sode s.

We can now turn to the principal matter in hand. The actions of $G$
on $M$ and $TM$ induce also a $G$-action on $TTM$: $\psiTTM: G \times
TTM \to TTM, (g,X_v) \mapsto T_v \psiTM_g(X_v)$.  As before, this
means that there exists a principal fibre bundle structure $\pTTM: TTM
\to TTM/G$ with the properties $\pTTM\circ \psiTTM_g = \pTTM$ and
$T\pTM\circ \psiTTM_g = T\pTM$.  Again, the fibres of $TTM$ and
$TTM/G$ are isomorphic, so $TTM \simeq (\pTM)^* TTM/G$.  Therefore,
the maps $\equivcl{T\pTM}: TTM/G \to T(TM/G), \equivcl{X} \mapsto
T\pTM(X)$ and $\equivcl{TT\pM}:TTM/G \to TT(M/G), \equivcl{X} \mapsto
TT\pM(X)$ are well-defined and lead to the following commutative
diagram:

\setlength{\unitlength}{1.2cm}
\begin{picture}(14.14,5)(-1.5,5.5)
\put(5.25,9.6){\vector(0,-1){1.6}}
\put(4.8,7.4){\vector(-3,-2){1}}
\put(4.8,9.6){\vector(-1,-2){1.4}} \put(5.7,7.4){\vector(3,-2){1}}
\put(5.7,9.60){\vector(1,-2){1.4}}
\put(4.08,6.4){\vector(1,0){2.71}} \put(5.25,10){\mybox{$TTM$}}
\put(3,6.4){\mybox{$T(TM/G)$}} \put(7.9,6.4){\mybox{$TT(M/G)$}}
\put(5.25,7.6){\mybox{$TTM/G$}}

\put(5.25,6){\mybox{$T\varrho$}}
\put(4.95,7.1){\mybox{\tiny{$\equivcl{T\pTM}$}}}
\put(5.85,6.85){\mybox{\tiny{$\equivcl{TT\pM}$}}}

\put(5.74,8.6){\mybox{$\pTTM$}}
\put(3.8,8.6){\mybox{$T\pTM$}}\put(6.9,8.6){\mybox{$TT\pM$}}

\end{picture}

The above diagram contains bundles over $TM$, $TM/G$ and $T(M/G)$.
All of them are Lie algebroids.  For example, $TTM/G$ is the Atiyah
algebroid of the manifold $TM$.  In addition, all the maps in the
diagram are Lie algebroid morphisms.

First, we will consider the outer triangle consisting of the spaces
$TTM$, $T(TM/G)$ and $TT(M/G)$.  Since $\pTM$ is a principal fibre
bundle, the kernel of $T\pTM$ can be identified with $TM\times \la$,
by means of the identification $(v,\xi)\mapsto \xi_{TM}(v)=
\clift{\xi_M}(v)$.  The two bases $\{X_i, {\tilde E}_a\}$ and
$\{X_i, {\hat E}_a\}$ of $\vectorfields{M}$, based on the moving and
the body-fixed basis, can be used to construct the basis
$\{\clift{X_i}, \clift{{\tilde E}_a}, \vlift{X_i}, \vlift{{\tilde
E}_a} \}$ of $\vectorfields{TM}$, which we call the standard basis,
and also the basis $\{ \clift{X_i}, \clift{{\tilde E}_a},
\vlift{X_i}, \vlift{{\hat E}_a} \}$, which we call the mixed basis.
Since $\xi_{TM} = \clift{\xi_M}$, it is clear that the vector fields
$\clift{{\tilde E}_a}$ span the vertical subbundle of the projection
$\pTM$.  The advantage of the mixed basis over the standard basis is
that the vector fields $\clift{X_i}$, $\vlift{X_i}$ and
$\vlift{{\hat E}_a}$ are all invariant:
\begin{eqnarray*}
\, [\clift{{\tilde E}_a}, \clift{X_i}] = \clift{[{\tilde E}_a, X_i]} = 0 , & &
[\clift{{\tilde E}_a}, \vlift{X_i}] = \vlift{[{\tilde E}_a, X_i]} = 0 \\
\,[\clift{{\tilde E}_a}, \vlift{{\hat E}_b}] =
\vlift{[{\tilde E}_a, {\hat E}_b ]} = 0,
&& [\clift{{\tilde E}_a}, \clift{{\tilde E}_b}] =
\clift{[{\tilde E}_a, {\tilde E}_b]} = C^c_{ab} \clift{{\tilde E}_{c}}.
\end{eqnarray*}
So the vector fields $\clift{X_i}, \vlift{X_i}$ and
$\vlift{\hat{E}_a}$ can be projected to sections of $TTM/G$ (by
means of $\pTTM$) and also to vector fields on $TM/G$ (by means of
$T\pTM$); the latter, denoted by $\ov{\clift{X_i}}$,
$\ov{\vlift{X_i}}$ and $\ov{\vlift{E_a}}$, form a basis of
$\vectorfields{TM/G}$. The following remark may be of some interest.
Observe that $TM/G \to M/G$ is a vector bundle; one can therefore
define a vertical lift operation taking sections of $TM/G \to M/G$
to vertical vector fields on $TM/G$.  The vector fields
$\ov{\vlift{X_i}}$ and $\ov{\vlift{E_a}}$ on $TM/G$ are in fact the
vertical lifts of the sections ${\ov X}_i, {\ov E}_a \in
\Sec(TM/G)$; so we could write $\ov{\vlift{X_i}}=\vlift{\ov{X}_i}$
and $\ov{\vlift{E_a}}=\vlift{\ov{E}_a}$.

The vector fields $\{\clift{\tilde{E}_a},\vlift{\tilde{E}_a}\}$ span the
vertical subbundle of the projection $T\pM$; so the kernel of
$TT\pM$ is isomorphic to $TM\times T\la \simeq TM\times
\la\times\la$, the isomorphism being
$X^a\clift{\tilde{E}_a} + Z^a\vlift{\tilde{E}_a} \mapsto (X^a E_a,Z^aE_a)$.
The vector fields $\{\vlift{\tilde{E}_a}\}$ span the kernel of the projection
$\varrho$; so $\Ker T\varrho$ is isomorphic to $TM\times\la$, by
$Z^a\vlift{\tilde{E}_a} \mapsto Z^aE_a$.

In this way we arrive at the following diagram of short exact
sequences (taking into account the fact that $(T\pM)^* TT(M/G) =
(\pTM)^*\varrho^*TT(M/G)$):

\setlength{\unitlength}{1cm}
\begin{picture}(14.29,6.5)(-3,4)
\put(5.3,9.9){\vector(1,0){2.2}} \put(1.3,7.3){\vector(1,0){2.0}}
\put(5.1,7.3){\vector(1,0){2}} \put(1.14,4.5){\vector(1,0){1.5}}
\put(5.8,4.5){\vector(1,0){1.1}} \put(1.3,9.9){\vector(1,0){1.8}}

\put(8.7,9.5){\vector(0,-1){1.7}} \put(4.2,6.8){\vector(0,-1){1.7}}
\put(4.2,9.5){\vector(0,-1){1.7}} \put(0.6,6.8){\vector(0,-1){1.7}}
\put(0.6,9.5){\vector(0,-1){1.7}} \put(8.7,6.8){\vector(0,-1){1.7}}

\put(0.4,9.9){\mybox{\fbox{$TM\times\la$}}}
\put(4.2,9.9){\mybox{$TM\times \la\times \la$}}
\put(8.7,9.9){\mybox{\fbox{$TM\times \la$}}}

\put(0.5,7.3){\mybox{$TM\times \la$}}
\put(4.2,7.3){\mybox{\fbox{$TTM$}}}
\put(8.7,7.3){\mybox{$(\pTM)^*T(TM/G)$}}

\put(0.6,4.5){\mybox{$\fbox{0}$}}
\put(4.2,4.2){\mybox{$(\pTM)^*\varrho^*TT(M/G)$}}
\put(4.2,4.8){\mybox{$(T\pM)^*TT(M/G)$}} \put(4.2,4.5){\mybox{=}}
\put(8.7,4.5){\mybox{\fbox{$(\pTM)^*\varrho^*TT(M/G)$}}}

\put(6.2,7.6){\mybox{$T\pTM$}} \put(4.9,6){\mybox{$TT\pM$}}
\put(9.6,6){\mybox{$(\pTM)^*T\varrho$}}

\end{picture}

There is a similar diagram for the spaces over $M/G$:

\setlength{\unitlength}{1cm}
\begin{picture}(14.29,6.5)(-3,4)
\put(5.1,9.9){\vector(1,0){2.0}} \put(1.14,7.3){\vector(1,0){2.0}}
\put(5.1,7.3){\vector(1,0){2.0}} \put(1.14,4.5){\vector(1,0){1.9}}
\put(5.3,4.5){\vector(1,0){1.6}} \put(1.14,9.9){\vector(1,0){2.0}}

\put(8.13,9.5){\vector(0,-1){1.7}}
\put(4.2,6.8){\vector(0,-1){1.7}}
\put(4.2,9.5){\vector(0,-1){1.7}}
\put(0.6,6.8){\vector(0,-1){1.7}}
\put(0.6,9.5){\vector(0,-1){1.7}}
\put(8.13,6.8){\vector(0,-1){1.7}}

\put(0.6,9.9){\mybox{\fbox{$\otau^*\og$}}}
\put(4.2,9.9){\mybox{$\otau^*(\og\times\og)$}}
\put(8.13,9.9){\mybox{\fbox{$\otau^*\og$}}}

\put(0.6,7.3){\mybox{$\otau^*\og$}}
\put(4.2,7.3){\mybox{\fbox{$TTM/G$}}}
\put(8.13,7.3){\mybox{$T(TM/G)$}}

\put(0.6,4.5){\mybox{\fbox{$0$}}}
\put(4.2,4.5){\mybox{$\varrho^*TT(M/G)$}}
\put(8.13,4.5){\mybox{\fbox{$\varrho^*TT(M/G)$}}}

\put(6.2,7.6){\mybox{$\equivcl{T\pTM}$}} \put(5,6){\mybox{$\equivcl{TT\pM}$}}
\put(8.5,6){\mybox{$T\varrho$}}

\end{picture}

Recall that $\ov\tau$ is the projection $TM/G \to M/G$.  The
identification $(TM\times \la)/ G \simeq {\ov\tau}^*\og $ is given
explicitly by $\equivcl{v_m,\xi} \mapsto (\equivcl{v_m},\equivcl{m,\xi})$.

We will use the basis $\{\fpd{}{x^i},\fpd{}{v^i}\}$ for
$\vectorfields{T(M/G)}$.  The basic sections
$\{\fpd{}{x^i},\fpd{}{v^i}\}$ can also be used for bases of vector
fields along the projection in the case of certain pull-back bundles,
that is both for $\Sec((T\pM)^*TT(M/G))$ and $\Sec(\varrho^*TT(M/G))$.
Finally, we will also use $\{\ov{\clift{X_i}},
\ov{\vlift{X_i}},\ov{\vlift{E_a}}\}$ as a basis for
$\Sec((\pTM)^*T(TM/G))$.

The two square commutative diagrams above will play the same role as
the short exact sequences (\ref{short1}) and (\ref{short2}) in the
first-order case.  We will show that a principal connection on $M$
induces splittings for all the short exact sequences in the squares.
The idea is that connections of the $M$ square which are
$G$-invariant in the appropriate sense automatically give rise to
connections for the $M/G$ square and vice versa. We shall explicitly
construct connections in both the middle horizontal and vertical
sequences in the $M$ square diagram, and use the general results at
the beginning of this section to complete the task.

The first induced connection lives on the middle horizontal line of
the $M$ square diagram; it is the so-called vertical lift of the
principal connection on $M$.

\begin{prop}
Suppose given a connection on a principal $G$-bundle $\pM:M\to M/G$,
specified by its connection form $\varpi$. The pull-back
$\tau^*\varpi$ of $\varpi$ to $TM$ (where $\tau:TM\to M$ is the
tangent bundle projection) is the connection form of a principal
connection on the principal $G$-bundle $TM\to TM/G$.
\end{prop}
\begin{proof}
 Clearly,
$\tau^*\varpi$ is a $\la$-valued 1-form on $TM$.  The action
$\psiTM$ of $G$ on $TM$ is $\tau$-related to the action $\psiM$ on
$M$. Moreover, the fundamental vector fields corresponding to the
two actions are related by $\xi_{TM}=\clift{(\xi_M)}$ for any
$\xi\in\la$; and in particular $T\tau(\xi_{TM})=\xi_M$.  Thus
\[
\tau^*\varpi(\xi_{TM})=\varpi(T\tau(\xi_{TM}))
=\varpi(\xi_M)=\xi,
\]
while
\[
\psiTM_g{}^*\tau^*\varpi=\tau^*\psiM_g{}^*\varpi=\ad_{g^{-1}}\tau^*\varpi,
\]
as required.\end{proof}
 The connection defined by $\tau^*\varpi$ is
called the vertical lift of the original connection; its right and
left splittings are denoted by $\vlift{\gamma}$ and $\vlift{\omega}$
(so that the connection form $\vlift{\varpi}$ is just given by
$\vlift{\varpi}=\tau^*\varpi$). The right splitting $\vlift{\gamma}$
at the level of the $M$ square can be given as follows.  Let
$\pTM(v)=\tilde{v}$; then
\[
\vlift{\gamma}: (\pTM)^*T(TM/G) \to TTM,\qquad (v,X_{\tilde v})
\mapsto W,
\]
where $W$ is determined by the condition $T\pTM(W) = X_{\tilde v}$
and $T\tau(W)=\gamma(m,T\ov\tau(X_{\tilde v}))$, where $m=\tau(v)$.
The first conditions shows that the above defines a splitting. To
see that it is the one that corresponds with the vertical lift
connection, we give the actions of $\vlift{\gamma}$ and
$\vlift{\omega}$ on the basis vector fields. We have
$T\tau(\clift{X_i}) = X_i \circ \tau$, $T\tau(\vlift{X_i}) = 0$ and
$T\tau(\vlift{\hat{E}_a}) = 0$. Likewise,
$T\ov\tau(\ov{\clift{X_i}}) = \fpd{}{x^i} \circ \ov\tau$,
$T\ov\tau(\ov{\vlift{X_i}}) = 0$ and $T\ov\tau(\ov{\vlift{E_a}}) =
0$, where $\ov\tau: TM/ G \to M/G$.  It follows that
\[
\vlift{\gamma}(\ov{\clift{X_i}}) = \clift{X_i}, \quad
\vlift{\gamma}(\ov{\vlift{X}_i}) = \vlift{X_i}, \quad
\vlift{\gamma}(\ov{\vlift{E_a}}) = \vlift{\hat{E}_a},
\]
and from these formulas we get for the associated left splitting
\[
\vlift{\omega}(\clift{X_i}) = 0, \quad
\vlift{\omega}(\vlift{X_i}) = 0, \quad
\vlift{\omega}(\clift{\tilde{E}_a}) = \clift{\tilde{E}_a} \quad \mbox{and}\quad
\vlift{\omega}(\vlift{\hat{E}_a}) = 0=\vomega(\vlift{\tilde{E}_a}).
\]
From the above relations it is clear that
$\vlift{\varpi}=\tau^*\varpi$, and that
therefore $\vlift{\gamma}$ is indeed the right splitting corresponding
to the vertical lift of the principal connection on $M$ specified by
$\varpi$, as it was defined initially.

One can find a connection whose definition is somewhat similar to
that of the vertical lift connection in \cite{MMR}, albeit in a much
less general context: the authors deal only with a Lagrangian system
with symmetry and restrict the dynamics to a particular value of the
momentum map.

The second connection of interest is a connection on the middle
vertical line in the $M$ square diagram, that is, it is a connection
on the bundle $T\pM: TM \to T(M/G)$. It is in fact a particular case
of a quite general construction which can be described as follows.

We first make an obvious remark.  The complete lift operation
$\vectorfields{M}\to\vectorfields{TM}$, $X\mapsto \clift{X}$, is not
$C^\infty(M)$-linear:\ in fact for a function $f$ on $M$ we have
$\clift{(fX)}=f\clift{X}+\dot{f}\vlift{X}$, where $\dot{f}$ is the
total derivative of $f$; the point to note is that $\clift{(fX)}$ is a
$C^\infty(TM)$-linear combination of $\clift{X}$ and $\vlift{X}$.
Suppose now that $M$ is equipped with a distribution (vector field
system) $\D$.  Let $\{X_i\}$ be a local vector field basis for $\D$,
and consider the local vector fields $\{\clift{X_i},\vlift{X_j}\}$ on
$TM$:\ they are linearly independent, and there are $2\dim\D$ of them.
Furthermore, if $\{Y_i\}$ is another local basis for $\D$ then the
span of $\{\clift{Y_i},\vlift{Y_j}\}$ coincides with the span of
$\{\clift{X_i},\vlift{X_j}\}$, as follows from the observation above
about $\clift{(fX)}$.  The span of $\{\clift{X_i},\vlift{X_j}\}$,
where $\{X_i\}$ is any local basis of $\D$, accordingly defines a
$2\dim\D$-dimensional distribution $\D'$ on $TM$.  Suppose next that
$\phi$ is a diffeomorphism of $M$ and $\clift{\phi}$ is the induced
diffeomorphism of $TM$.  Denote by $\phi^\X$ the action of $\phi$ on
vector fields on $M$, ${\clift{\phi}}^\X$ the action of $\clift{\phi}$
on vector fields on $TM$.  Then
$\clift{\phi^\X(X)}={\clift{\phi}}^\X(\clift{X})$ and
$\vlift{\phi^\X(X)}={\clift{\phi}}^\X(\vlift{X})$ (these are the
integrated versions of two formulas for brackets between complete and
vertical lifts which we used earlier).  Thus if $\D$ is invariant
under the action of some group $G$ on $M$ then $\D'$ is invariant
under the induced action of $G$ on $TM$.  Now let $M\to M/G$ be a
principal $G$-bundle and $\D$ the horizontal distribution of a
principal connection:\ then $\D'$ is a $G$-invariant distribution on
$TM$ which is transverse to the fibres of $TM\to T(M/G)$, that is, a
connection on $TM\to T(M/G)$, which is $G$-invariant in the
appropriate sense.

It is easy to describe the left splitting of the new connection, as
follows.

The complete lift construction can be extended from vector fields to
tensor fields, as is shown in \cite{TM}. In particular, given a type
$(1,1)$ tensor field $A$ on a manifold $M$, its complete lift
$\clift{A}$ is a type $(1,1)$ tensor field on $TM$ with the following
properties:
\[
\clift{A}(\vlift{X})=\vlift{A(X)},\quad
\clift{A}(\clift{X})=\clift{A(X)},\quad
\lie{\clift{X}}\clift{A}=\clift{(\lie{X}A)},
\]
for any vector field $X$ on $M$. Moreover, for any two type
$(1,1)$ tensor fields $A$, $B$ on $M$,
$\clift{A}\clift{B}=\clift{(AB)}$. The complete lift $\clift{A}$ may
be described explicitly as follows. Regard $A$ as a fibre-linear map
$TM\to TM$, fibred over the identity. Let $\sigma: TTM \to TTM$
denote the canonical involution:\ then $\clift{A}$, regarded as a
fibre-linear map $TTM\to TTM$, is given by $\clift{A}=\sigma\circ
TA\circ\sigma$ (where $TA$ is the tangent map, or differential, of
the map $A$).

\begin{prop}
Consider the left splitting $\omega:TM \to M\times\la$ of a
principal connection on $\pM:M\to M/G$. The complete lift
$\clift\omega$ of $\omega$ defines a $G$-invariant connection on
$TM\to T(M/G)$.
\end{prop}
\begin{proof}
As we pointed out earlier, $\omega$ can be considered as a type
$(1,1)$ tensor field on $M$, when we regard $M\times\la$ as a
subbundle of $TM$; from this point of view, for each $m\in M$,
$\omega_m$ is the projection onto the vertical subspace of $T_mM$
along the horizontal subspace; since it is a projection operator
$\omega$ satisfies $\omega^2=\omega$.  The fact that the connection
is principal is equivalent to the fact that, as a type $(1,1)$
tensor, $\omega$ is $G$-invariant, which is to say that
$\omega\circ\psi^{\mathcal{X}}_g=\psi^{\mathcal{X}}_g\circ\omega$
for all $g\in G$, where $\psi^{\mathcal{X}}$ is the $G$-action on
vector fields.  When $G$ is connected the latter condition is
equivalent to $\lie{\xi_M}\omega=0$ for all $\xi\in\la$.  We take
the complete lift $\clift{\omega}$, to obtain a type $(1,1)$ tensor
field on $TM$.  Now
$(\clift{\omega})^2=\clift{(\omega^2)}=\clift{\omega}$, so
$\clift{\omega}$ is a projection operator.  From the formulas for
the action of $\clift{\omega}$ on vertical and complete lifts it is
clear that it vanishes on vertical and complete lifts of vector
fields which are horizontal with respect to $\omega$, that is, on
$\D'$.  Moreover, for any $\xi\in\la$ we have
\[
\clift{\omega}(\vlift{\xi_M})=\vlift{(\omega(\xi_M))}
=\vlift{\xi_M},\quad
\clift{\omega}(\clift{\xi_M})=\clift{(\omega(\xi_M))}
=\clift{\xi_M},
\]
so that $\Im(\clift{\omega})$ can be identified with $\la\times\la$
in the required manner.

 Finally, we have
\[
\lie{\xi_{TM}}\clift{\omega}=
\lie{\clift{\xi_M}}\clift{\omega}=\clift{(\lie{\xi_M}\omega)}=0,
\]
which expresses the $G$-invariance of $\clift{\omega}$.\end{proof}

The connection determined by $\comega$ was first described, in its
essentials, by Vilms \cite{Vilms}.  In fact it was shown in
\cite{Vilms} that a connection on a vector bundle $E\to X$ induces a
connection on the bundle $TE \to TX$.  Of course the bundle $\pM: M
\to M/G$ we are dealing with in the current situation is not a
vector bundle; nevertheless, Vilms's result may be extended to cover
it.  We therefore call this connection the Vilms connection;
however, we denote its splittings by $\cgamma$ and $\comega$ (as
before).

Note the important but somewhat subtle difference between the
constructions of the two connections:\ in constructing the vertical
lift connection we specify the initial connection by the $\la$-valued
connection form $\varpi$, but in constructing the Vilms connection we
specify it by the type $(1,1)$ tensor field $\omega$ defining the
right splitting.  We mention this because there is a concept of the
vertical lift of a type $(1,1)$ tensor field, and it is important to
realise that we do not use this concept here.

The right splitting $\cgamma$ of the Vilms connection is a map
$(T\pM)^*TT(M/G) \to TTM$, which may be specified as follows.  We
denote by $\ov\sigma$ the canonical involution of $TT(M/G)$.  Let $T\pM (v) =
\ov v$. The right splitting of the Vilms connection is given by
\[
\cgamma: (v, Y_{\ov v}) \mapsto \sigma\left(T\gamma(v,
{\ov\sigma} (Y_{\ov v})) \right).
\]

The fact that this is a splitting is due to the property $TT\pM
\circ \sigma = {\ov\sigma} \circ TT\pM$, as it is easy to see.  Indeed,
\[
TT\pM \circ\cgamma (v,Y) = {\ov\sigma} \circ T(T\pM\circ
\gamma)(v,{\ov\sigma}(Y)) = \ov\sigma \circ \ov\sigma (Y)=Y.
\]
We calculate the corresponding right splitting, and confirm that it
is $\comega$. The right splitting is given by $\id - \cgamma \circ
TT\pM$; we have
\begin{eqnarray*}
\id - \cgamma \circ TT\pM
&=& \id -\sigma\circ T\gamma \circ \ov\sigma \circ TT\pM\\
&=& \id - \sigma\circ T\gamma \circ  TT\pM \circ \sigma
=\id - \sigma \circ T(\gamma\circ T\pM) \circ \sigma\\
&=& \id - \sigma \circ T(\id -\omega)\circ \sigma
= \sigma \circ T \omega \circ \sigma=\comega.
\end{eqnarray*}
In terms of the standard basis
$\{\clift{X_i}, \clift{\tilde{E}_a}, \vlift{X_i},
\vlift{\tilde{E}_a}\}$ we have
\[
\comega(\clift{X_i}) = 0, \quad
\comega(\vlift{X_i}) = 0, \quad
\comega(\clift{\tilde{E}_a}) = \clift{\tilde{E}_a} \quad\mbox{and}\quad
\comega(\vlift{\tilde{E}_a}) = \vlift{\tilde{E}_a};
\]
equally, $\comega(\vlift{\hat{E}_a}) = \vlift{\hat{E}_a}$.
Since $TT\pM(\clift{X_i}) = \fpd{}{x^i} \circ T\pM$ and
$TT\pM(\vlift{X_i}) = \fpd{}{v^i} \circ T\pM$, it also follows that
\[
\cgamma\left(\fpd{}{x^i}\right) = \clift{X_i}
-\comega(\clift{X_i}) = \clift{X_i}
\quad \mbox{and} \quad
\cgamma\left(\fpd{}{v^i}\right) = \vlift{X_i}
-\comega(\vlift{X_i}) = \vlift{X_i}.
\]

We noted above that because of the invariance of $\omega$, the right
splitting $\comega$ of the Vilms connection is invariant.  It
follows that the Vilms connection can be quotiented to give a
connection on the $M/G$-square.

As a consequence of the existence of the two connections described so
far, we can deduce for both square diagrams a third connection.
Clearly $\Im(\cgamma)\subset\Im(\vlift{\gamma})$, so there is a
connection $\gamma'$ such that $\cgamma =\vlift{\gamma} \circ
\gamma'$, and a connection $\ov{\gamma'}$ such that and $\ov{\cgamma}
= \ov{\vgamma} \circ \ov{\gamma'}$.  For the appropriate bases
\[
\ov{\gamma'}(\fpd{}{x^i}) = \ov{\clift{X_i}} \qquad \mbox{and} \qquad
\ov{\gamma'}(\fpd{}{v^i}) = \ov{\vlift{X_i}},
\]
and
\[
\ov{\omega'}(\ov{\clift{X_i}}) = 0, \quad
\ov{\omega'}(\ov{\vlift{X_i}}) = 0
\quad \mbox{and} \quad
\ov{\omega'}(\ov{\vlift{E_a}}) = \ov{\vlift{E_a}}.
\]

\section{Second-order systems}

We now come to the consideration of second-order systems.  We assume
given a \sode, that is, a vector field $\Gamma$ on $TM$ such that
$T\tau\Gamma(v)=v$ for all $v\in TM$, where $\tau: TM\to M$ is the
tangent bundle projection.  Furthermore, we assume that $\Gamma$ is
$G$-invariant, so that it satisfies $\psiTTM_g\Gamma(v) =
\Gamma(\psiTM_gv)$.  There is therefore a section $\ov\Gamma$ of
$TTM/G$ such that $\pTTM\circ \Gamma = \ov\Gamma \circ \pTM$.  Under
the appropriate maps $\Gamma$ projects onto $\Gamma_1$ and $\Gamma_2$
as shown below, and can be decomposed into elements which are boxed in
the diagram.  Analogously, $\ov\Gamma$ projects on ${\ov\Gamma}_1$ and
${\ov\Gamma}_2$, and has a similar decomposition.  Of course, all the
elements are related to each other in an appropriate way; for example
$\Gamma_1\in \Sec((\pTM)^*T(TM/G))$, which is given by $\Gamma_1(v) =
(v, T\pTM(\Gamma(v)))$, can be reduced to the vector field
${\ov\Gamma}_1\in \vectorfields{TM/G}$.

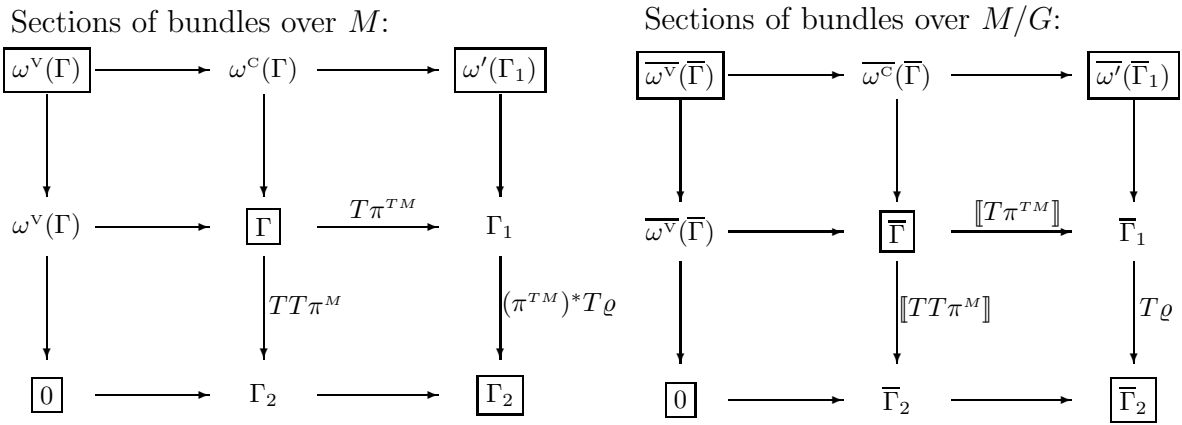
\begin{tabular}{ll}
\begin{minipage}{8cm}

Sections of bundles over $M$:

\setlength{\unitlength}{0.8cm}
\begin{picture}(14.29,6.5)(0,4)
\put(5.1,9.9){\vector(1,0){2.0}} \put(1.4,7.3){\vector(1,0){1.9}}
\put(5.1,7.3){\vector(1,0){2.0}} \put(1.4,4.5){\vector(1,0){1.9}}
\put(5.1,4.5){\vector(1,0){2.0}} \put(1.4,9.9){\vector(1,0){1.9}}

\put(8.13,9.5){\vector(0,-1){1.7}}
\put(4.2,6.8){\vector(0,-1){1.7}}
\put(4.2,9.5){\vector(0,-1){1.7}}
\put(0.6,6.8){\vector(0,-1){1.7}}
\put(0.6,9.5){\vector(0,-1){1.7}}
\put(8.13,6.8){\vector(0,-1){1.7}}

\put(0.6,9.9){\mybox{\fbox{$\vomega(\Gamma)$}}}
\put(4.2,9.9){\mybox{$\comega(\Gamma)$}}
\put(8.13,9.9){\mybox{\fbox{$\omega'(\Gamma_1)$}}}

\put(0.6,7.3){\mybox{$\vomega(\Gamma)$}}
\put(4.2,7.3){\mybox{\fbox{$\Gamma$}}}
\put(8.13,7.3){\mybox{$\Gamma_1$}}

\put(0.6,4.5){\mybox{$\fbox{0}$}}
\put(4.2,4.5){\mybox{$\Gamma_2$}}
\put(8.13,4.5){\mybox{\fbox{$\Gamma_2$}}}

\put(6.2,7.6){\mybox{$T\pTM$}} \put(4.9,6){\mybox{$TT\pM$}}
\put(9.1,6){\mybox{$(\pTM)^*T\varrho$}}

\end{picture}

\end{minipage}

&

\begin{minipage}{9cm}

Sections of bundles over $M/G$:

\setlength{\unitlength}{0.8cm}
\begin{picture}(14.29,6.5)(0,4)
\put(5.1,9.9){\vector(1,0){2.0}} \put(1.4,7.3){\vector(1,0){1.9}}
\put(5.1,7.3){\vector(1,0){2.0}} \put(1.4,4.5){\vector(1,0){1.9}}
\put(5.1,4.5){\vector(1,0){2.0}} \put(1.4,9.9){\vector(1,0){1.9}}

\put(8.13,9.5){\vector(0,-1){1.7}}
\put(4.2,6.8){\vector(0,-1){1.7}}
\put(4.2,9.5){\vector(0,-1){1.7}}
\put(0.6,6.8){\vector(0,-1){1.7}}
\put(0.6,9.5){\vector(0,-1){1.7}}
\put(8.13,6.8){\vector(0,-1){1.7}}

\put(0.6,9.9){\mybox{\fbox{$\ov{\vomega}(\ov\Gamma)$}}}
\put(4.2,9.9){\mybox{$\ov{\comega}(\ov\Gamma)$}}
\put(8.13,9.9){\mybox{\fbox{$\ov{\omega'}(\ov{\Gamma}_1)$}}}

\put(0.6,7.3){\mybox{$\ov{\vomega}(\ov\Gamma)$}}
\put(4.2,7.3){\mybox{\fbox{$\ov\Gamma$}}}
\put(8.13,7.3){\mybox{${\ov\Gamma}_1$}}

\put(0.6,4.5){\mybox{$\fbox{0}$}}
\put(4.2,4.5){\mybox{${\ov\Gamma}_2$}}
\put(8.13,4.5){\mybox{\fbox{${\ov\Gamma}_2$}}}

\put(6.2,7.6){\mybox{$\equivcl{T\pTM}$}} \put(5,6){\mybox{$\equivcl{TT\pM}$}}
\put(8.5,6){\mybox{$T\varrho$}}

\end{picture}

\end{minipage}

\end{tabular}

In fact the connections give a $G$-invariant decomposition of $\Gamma$
into three parts:
\begin{eqnarray*}
\Gamma &=& \vgamma(\Gamma_1) + \vomega(\Gamma)\\
&=&\cgamma(\Gamma_2) + \comega(\Gamma)\\
&=&\cgamma(\Gamma_2)
+ \vgamma(\omega'(\Gamma_1)) + \vomega(\Gamma).
\end{eqnarray*}
Since $\Gamma$ is a \sode\ the $\la$-valued function $\vlift{\varpi}(\Gamma)$
on $TTM$ is given by
\[
\vlift{\varpi}(\Gamma(v))=\varpi(T\tau(\Gamma(v))=\varpi(v)
\]
for all $v\in TM$; thus (for a given choice of $\varpi$)
$\vlift{\varpi}(\Gamma)$ is the same for all \sode s $\Gamma$.

There is an analogous three-way decomposition of $\ov\Gamma$, as
shown in the right-hand diagram. Here ${\ov\Gamma}_1$ is a vector
field on $TM/G$, ${\ov\Gamma}_2$ is a section of $\varrho^*TT(M/G)$
(i.e.\ a vector field along $\varrho$) and
$\ov{\omega'}({\ov\Gamma}_1)\in \Sec(TM/G \times \og)$.

As well as being a section of $TTM/G\to TM/G$, $\ov\Gamma$ may also
be regarded as a section of the so-called prolongation bundle
$T^\varrho (TM/G)\to TM/G$, whose fibre at $\tilde v \in TM/G$ is
\[
T^\varrho_{\tilde v} (TM/G) = \{ (\tilde w, X_{\tilde v}  ) \in TM/G
\times T(TM/G) \mid \varrho(\tilde w) = T\ov\tau (X_{\tilde v})\}.
\]
Theorem 9.1 of \cite{DMM} shows that the identification of the
quotient bundle $TTM/G\to TM/G$ with the above bundle is given by
the isomorphism
\[
TTM/G \to T^\varrho (TM/G),\qquad \equivcl{W} \mapsto (\pTM
T\tau(W), T\pTM(W)).
\]
This map is independent of the choice of $W$ since for any other
$\psiTTM_g W$ within the same equivalence class, $T\tau(\psiTTM_g W)
= T (\tau \circ \psiTM_g) (W) =  T\tau (W)$.
The second component of the isomorphism is in fact the map
$\equivcl{T\pTM}:TTM/G \to T(TM/G)$. Keeping in mind that $\Gamma$
is a \sode, ${\ov\Gamma} \in \Sec(TTM/G)$ can be identified with the
section $\tilde v \mapsto (\tilde v, {\ov\Gamma}_1(\tilde v))$ of
the prolongation bundle. As a consequence, the composing parts of
this section satisfy $T{\ov\tau} ({\ov\Gamma}_1(\tilde v)) =
\varrho(\tilde v)$. This property clearly resembles the defining
property $T\tau(\Gamma(v))=v$ of a \sode; sections of the
prolongation bundle of the above form were therefore called `pseudo
second-order differential equation sections' in e.g.~\cite{Me} or
`second-order differential equations' in e.g.~\cite{DMM}.

Next, we will discuss the reconstruction process.  We will use the
following notations.  Let $v(t)\in TM$ denote an integral curve of
$\Gamma$ and let $c(t)$ be the corresponding base integral curve,
that is, $c(t) = \tau(v(t))\in M$.  It follows from the fact that
$\Gamma$ is a \sode\ that $v=\dot{c}$ (when we consider the latter
as a curve in $TM$).  We will write ${\tilde v}(t)= \pTM(v(t)) \in
TM/G$ and ${\ov c}(t) = \pM (c(t)) \in M/G$.  Obviously,
$\ov\tau(\tilde v) = \ov c$ and moreover $T\pM (v(t))=
\varrho({\tilde v}(t)) = {\dot{\ov c}}(t) \in T(M/G)$. In a previous
section we encountered the horizontal lift of ${\ov c}$ (with
respect to the connection on $\pM$), which we denoted by ${\ov
c}^\gamma$.

We first note that the vertical lift connection is a principal
connection on the principal fibre bundle $\pTM: TM \to TM/G$.  So
just as in the first-order case we can construct the integral curve
$t\mapsto v(t)$ of the invariant vector field
$\Gamma\in\vectorfields{TM}$ from an integral curve $t\mapsto{\tilde
v}(t)=\pTM(v(t))$ of the reduced vector field ${\ov\Gamma}_1 \in
\vectorfields{TM/G}$, by
\begin{itemize}
\item taking the horizontal lift $\tilde{v}^{\gamma}$ of
$\tilde{v}$ through $v(0)$ (with respect to $\vgamma$), and
\item finding the solution $t\mapsto g(t)\in G$ of the equation
\begin{equation} \label{omegaVeq}
\theta(\dot{g})=\vlift{\varpi}(\Gamma\circ\tilde{v}^{\gamma})
\end{equation}
with $g(0)=e$ (where $\theta$ is the Maurer-Cartan form of $G$);
\end{itemize}
the required integral curve is given by
\[
v(t)=\psiTM_{g(t)}\tilde{v}^{\gamma}(t).
\]
The right-hand side of equation~(\ref{omegaVeq}) can equally well be
written as $\varpi(\tilde{v}^{\gamma})$.

Let us look at the relation that determines the horizontal lift
${\tilde v}^\gamma$ of $\tilde{v}$ to $TM$. Let $\tilde v$ be a
given curve in $TM/G$ (not necessarily an integral curve of
${\ov\Gamma}_1$). By definition, ${\tilde v}^\gamma$ projects onto
$\tilde v$ and is a solution of
\[
{\dot{\tilde v}}^\gamma = \vgamma({\tilde v}^\gamma, \dot{\tilde
v}).
\]
From the conditions that determine $\vgamma$, we know that this is
equivalent with the properties $T\pTM\circ{\dot{\tilde v}}^\gamma =
\dot{\tilde v}$, and $T\tau\circ{\dot{\tilde v}}^\gamma = \gamma
(T\ov\tau\circ\dot{\tilde v})$. The first property simply recalls
that $\pTM(\tilde v^\gamma) = \tilde v$. If we denote as before
$\ov\tau\circ\tilde v = \ov c$, then $T\ov\tau\circ\dot{\tilde v}=
\dot{\ov c}$. So, we can deduce from the second property that
$T\tau\circ{\dot{\tilde v}}^\gamma = \dot{\ov{c}}^\gamma$. To
conclude, the curve ${\tilde v}^\gamma$ is completely determined by
the properties $\tau\circ{\tilde v}^\gamma=\ov{c}^\gamma$ and
$\pTM\circ{\tilde v}^\gamma = \tilde v$.

Any element of the vector bundle $TM/G\to M/G$ can be written as a
sum of two parts via the splitting $\ov{\gamma}$ of the Atiyah
sequence~(\ref{short2}):
\[
{\tilde v} = \ov\xi + \ov\gamma(\dot{\ov c}),
\]
where $\ov\xi= \ov\omega (\tilde v) \in \ov\la$. We can use this to
give a more explicit formulation of ${\tilde v}^\gamma$. Let $\xi(t)
\in \la$ be such that $\ov\xi = \equivcl{{\ov c}^\gamma, \xi}$. Then
the curve $\xi_M\circ{\ov c}^\gamma + \dot{{\ov c}}^\gamma$ in $TM$
projects onto ${\ov c}^\gamma$ by means of $\tau$ and projects onto
${\tilde v} = \ov\xi + \ov\gamma(\dot{\ov c})$ by means of $\pTM$;
so it can only be ${\tilde v}^\gamma$. Therefore, $\varpi({\tilde
v}^\gamma)=\xi$ and the horizontal part of $\tilde{v}^\gamma$ is
${\dot{\ov c}}^\gamma$.

Suppose now again that $\tilde v(t)$ is an integral curve of
${\ov\Gamma}_1$:
\begin{equation}\label{intcurveGamma1}
\dot{\tilde v} = {\ov\Gamma}_1\circ\tilde{v};
\end{equation}
and that $v(t) = \psiTM_{g(t)} {\tilde v}^\gamma(t)$ is an integral
curve of $\Gamma$. Notice that
\[
c=\tau\circ v=\tau\circ\psiTM_{g}\tilde{v}^{\gamma}
=\psiM_{g}(\tau\circ\tilde{v}^{\gamma})=\psiM_{g}\ov{c}^\gamma;
\]
that is to say, the curve in $G$ required to bring
$\tilde{v}^\gamma$ to $v$ in $TM$ is the same as the curve in $G$
required to bring $\ov{c}^\gamma$ to $c=\tau\circ v$ in $M$. In fact
from equation~(\ref{intcurveX}), since $\Gamma$ is a \sode,
\begin{equation}\label{vcdot}
v=\dot c= \psiTM_{g} \left((\theta(\dot{g}))_M\circ{\ov c}^\gamma +
\dot{{\ov c}}^\gamma\right).
\end{equation}
So, in the case that $\tilde v$ is an integral curve of
${\ov\Gamma}_1$, the curve $\xi(t)=\varpi({\tilde v}^\gamma(t)) \in
\la$ must equal $ \theta(\dot{g})$, which agrees with
equation~(\ref{omegaVeq}).

We turn finally to the integral curves $\tilde v$ of
${\ov\Gamma}_1$. We will use the connection $\ov{\omega'}$ to
decompose equation~(\ref{intcurveGamma1}) into two coupled equations
for the two curves $\ov\xi \in \ov\la$ and $\ov c\in M/G$ that
constitute $\tilde v$. The first equation is related to
${\ov\Gamma}_2\in \Sec(\varrho^*TT(M/G))$, which can be considered
as a map ${\ov\Gamma}_2: TM/G \to TT(M/G)$; thus
${\ov\Gamma}_2\circ\tilde{v}$ is a curve in $TT(M/G)$. The second
equation is related to $\ov{\omega'}({\ov\Gamma}_1)$, which is a
vertical vector field on $TM/G$.  This vector field can be regarded
as a section of $(\ov\tau)^*\og$, thus as a map $TM/G \to \og$; so
$\ov{\omega'}({\ov\Gamma}_1\circ\tilde{v})$ can be regarded as a
curve in $\og$.  The projection of this curve onto $M/G$ is
obviously ${\ov c}$.

We need to introduce one more concept: that of the associated linear
connection on the associated bundle $\ov\la \to M/G$ (see also
\cite{Cendra, Mac}). In fact, we will only need its covariant
derivative operator $\frac{D^A}{Dt}$ which acts on curves $\ov\xi$
in $\ov\la$. Let $\ov c(t)$ be the projection of $\ov\xi(t)$ on
$M/G$ and let $c(t)$ be any curve in $M$ that projects on $\ov
c(t)$. Let $\xi(t)\in \la$ be such that $\ov\xi = \equivcl{c,\xi}$.
Then, the covariant derivative of $\ov\xi$ can be defined as
\[
\frac{D^A \ov\xi}{Dt} = \equivcl{c, \dot{\xi}-[\varpi\circ {\dot c},
\xi]  }
\]
($\dot\xi$ stands here for the projection on the second argument of
this curve in $T\la=\la\times \la$). To see that this definition is
independent of the choice of the representative in the equivalence
class, take any other $d(t) \in M$ with $d(t) =\psiM_{h(t)}c(t)$.
The corresponding curve $\xi^d(t)$ in $\la$ such that
$\ov\xi=\equivcl{d, \xi^d}$ is then equal to $ad_{h^{-1}}\xi$.
Moreover, $\dot d = \psiTM_h \big(\dot c + (\theta(\dot h))(c)\big)$
and ${\dot\xi}_d = ad_{h^{-1}}\big(\dot\xi + [\theta(\dot h),
\xi]\big)$. So, indeed,
\begin{eqnarray*} \equivcl{d,
\dot{\xi^d}-[\varpi\circ{\dot d}, \xi^d]} & = & \equivcl{d,
ad_{h^{-1}}\big(\dot\xi + [\theta(\dot h), \xi]\big) -
[ad_{h^{-1}}\big(\varpi\circ\dot c + \theta(\dot h)\big),
ad_{h^{-1}}\xi] }
\\ & = &\equivcl{\psiM_{h}c, ad_{h^{-1}} \big(
\dot{\xi}-[\varpi\circ{\dot c}, \xi]\big) }= \equivcl{c,
\dot{\xi}-[\varpi\circ{\dot c}, \xi] }.
\end{eqnarray*}
Remark that in the particular case of the horizontal lift, $
\frac{D^A}{Dt}\equivcl{{\ov c}^\gamma, \xi} = \equivcl{{\ov
c}^\gamma, \dot{\xi}}$.

One can show that the explicit formula for the associated linear
connection is
\[
\nabla^A: \vectorfields{M/G} \times \Sec(\ov\la) \to \Sec(\ov\la):
(\ov X, \ov\xi)\mapsto \nabla^A_{\ov X} \ov\xi = [\ov\gamma(\ov
X),\ov\xi].
\]
Here $[\cdot,\cdot]$ stands for the above mentioned Lie algebroid
bracket of the Atiyah algebroid $TM/G$. For more details, see e.g.\
\cite{Mac}.

\begin{thm} Let $\ov\xi(t)\in
\ov\la$, $\ov c(t)\in M/G$ and put $\tilde v = \ov \xi + \ov\gamma
(\dot{\ov c})$.  If $\ov c$ and $\ov\xi$ are solutions of
\begin{equation} \label{decGamma1}
\left\{
\begin{array}{lll}
\ddot{\ov c} &=& {\ov\Gamma}_2 \circ{\tilde v},\\[1mm] \displaystyle
\frac{D^A{\ov\xi}}{Dt} &=& \ov{\omega'}({\ov\Gamma}_1\circ{\tilde
v}),
\end{array}
\right.
\end{equation}
then $\tilde v$ is an integral curve of ${\ov\Gamma}_1$. Solve
${\dot{\ov c}}^\gamma = \gamma({\ov c}^\gamma, \dot{\ov c})$ for
${\ov c}^\gamma(t) \in M$ and let $\xi(t)\in\la$ be such that
$\ov\xi = \equivcl{{\ov c}^\gamma,\xi}$. If $g(t)\in G$ is a
solution of
\begin{equation} \label{gxieq}
\theta(\dot{g})= \xi
\end{equation}
then the curve $v=\psiTM_g(\xi_M\circ\ov{c}^\gamma + {\dot{\ov
c}}^\gamma)= \psiTM_g {\tilde v}^\gamma$ is an integral curve of
$\Gamma$.

Conversely, suppose that $v$ is an integral curve of $\Gamma$. Let
$\tilde v = \pTM\circ v$, $\ov c= \pM\circ\tau\circ v$ and
$\ov\xi=\ov\omega \circ \tilde v$. Then ${\tilde v}$ is an integral
curve of ${\ov\Gamma}_1$ and $\ov c$ and $\ov \xi$ satisfy
(\ref{decGamma1}). Compute ${\ov c}^\gamma$ from ${\dot{\ov
c}}^\gamma = \gamma({\ov c}^\gamma, \dot{\ov c})$. Let $g\in G$ be
such that $c=\psiTM_g {\ov c}^\gamma$ and let $\xi\in\la$ be such
that $\ov\xi=\equivcl{{\ov c}^\gamma,\xi}$. Then $g$ satisfies
equation (\ref{gxieq}).
\end{thm}

If $d$ is a curve in $M$ such that $\pM(d) =\ov c$ and if we denote
by $\xi^d$ the curve in $\la$ which is such that
$\ov\xi=\equivcl{d,\xi}$, then the last equation of
(\ref{decGamma1}) could equivalently be written as $\dot\xi -
[\varpi(\dot d),\xi] = \omega'(\Gamma_1 \circ v^d)$, where $v^d =
\xi_M \circ d + \gamma(d,\dot{\ov c})$ is the unique curve on $TM$
that projects on both ${\tilde v}$ and $d$. Indeed, the relation
between $\omega'(\Gamma_1)$ and ${\ov\omega}'({\ov\Gamma}_1)$ is
 \[
{\ov\omega}'({\ov\Gamma}_1)(\pTM(v_m)) = \equivcl{m,
\omega'(\Gamma_1)(v_m)},
\]
and therefore, ${\ov\omega}'({\ov\Gamma}_1\circ {\tilde v}) =
\equivcl{d, \omega'(\Gamma_1 \circ {\tilde v}^d )}$.

If one is interested only in the coordinates on $M/G$ (`shape'
variables in \cite{Bloch}), it is necessary only to solve equations
(\ref{decGamma1}) where the symmetry has already been cancelled out.
If the whole motion on $M$ is required one will have to solve the
whole system.

The proof of theorem will follow from the considerations of the
coordinate version of the reduced second-order equations in the
following paragraphs.

We begin our description of the coordinate expression of the
equations with a general remark. If we take any local basis
$\{X_\alpha\}$ of vector fields on some manifold $M$, not
necessarily a coordinate basis, and express any tangent vector $v$
at $m\in M$ in terms of this basis so that $v=v^\alpha X_\alpha|_m$,
then the $v^\alpha$ will serve as fibre coordinates on $TM$.  A
vector field $\Gamma$ on $TM$ will be a \sode\ if and only if it
takes the form
$\Gamma=v^\alpha\clift{X_\alpha}+F^\alpha\vlift{X}_\alpha$, and its
integral curves will satisfy $\dot{v}^\alpha=F^\alpha$.  In terms of
a new basis $\{Y_\alpha\}$, where $Y_\alpha=A_\alpha^\beta X_\beta$,
we have $v=v^\alpha X_\alpha|_m=w^\alpha Y_\alpha|_m$ where $w^\beta
A^\alpha_\beta=v^\alpha$.  Moreover,
$\Gamma=w^\alpha\clift{Y_\alpha}+G^\alpha\vlift{Y}_\alpha$ where
$F^\alpha=A^\alpha_\beta G^\beta+\dot{A}^\alpha_\beta w^\beta$, the
overdot here indicating the total derivative.

We turn now to the case of interest. We have defined on $M$ two local vector
field bases $\{X_i,\tilde{E}_a\}$  and $\{X_i,\hat{E}_a\}$, with
\[
X_i=\gamma\left(\vf{x^i}\right)
\]
where the $x^i$ are local coordinates on $M/G$, and $\{E_a\}$ is a
basis for $\la$. Both $\{\tilde{E}_a\}$ and $\{\hat{E}_a\}$ are
bases of vector fields which are vertical with respect to the
projection $\pM:M\to M/G$. The first, which consists of fundamental
vector fields, we called the moving basis. The body-fixed local
basis $\{\hat{E}_a\}$, on the other hand, consists of $G$-invariant
vertical vector fields. We have $\hat{E}_a=A_a^b\tilde{E}_b$ where
the coefficients $A_a^b$ satisfy
$\tilde{E}_a(A_b^c)+C^c_{ad}A_b^d=0$ (equation~(\ref{eqnforA})).

For any $v\in T_mM$ we set
\[
v=v^iX_i|_m+v^a\tilde{E}_a|_m=v^iX_i|_m+w^a\hat{E}_a|_m;\quad
v^a=A^a_bw^b.
\]
The $v^i$ may be regarded as the fibre coordinates on $T(M/G)$
corresponding to the base coordinates $x^i$. We show in the following
paragraph that the $v^a$ satisfy
\begin{equation}\label{vacoords}
\clift{\tilde{E}_b}(v^a)+C^a_{bc}v^c=0;
\end{equation}
so from equation~(\ref{secassoc}) we may consider $v\mapsto v^aE_a$ as
defining a section of $\ov{\tau}^*\ov{\la}\to TM/G$; this is just the
section $\ov{\vomega}(\ov{\Gamma})$. On the other hand, since
$v^a=A^a_bw^b$, where the $A^a_b$ are functions on $M$, we have
\[
A^a_c\clift{\tilde{E}_b}(w^c)+\tilde{E}_b(A^a_c)w^c+C^a_{bc}v^c=0;
\]
but from equation (\ref{eqnforA})
\[
\tilde{E}_b(A^a_c)w^c=-C_{bd}^aA^d_cw^c=-C^a_{bc}v^c,
\]
whence $\clift{\tilde{E}_b}(w^a)=0$, as one might have expected.

Equation~(\ref{vacoords}) is a consequence of the following general
considerations. Let $\{Z_\alpha\}$ be any local basis of vector
fields on a manifold $M$, with dual basis of 1-forms
$\theta^\alpha$.  Let $\hat{\theta}$ be the fibre-linear function on
$TM$ defined by a 1-form $\theta$ on $M$, so that
$\hat{\theta}(x,v)=\theta_x(v)$.  Then if
$v^\alpha=\hat{\theta}^\alpha(x,v)$, $v=v^\alpha Z_\alpha|_x$.  For
any vector field $Z$ on $M$,
$\clift{Z}(\hat{\theta})=\widehat{\lie{Z}\theta}$.  But
$\lie{Z^\alpha}{\theta^\beta}(Z_\gamma)=-\theta^\beta([Z_\alpha,Z_\gamma])$,
so
$\lie{Z^\alpha}{\theta^\beta}=-C^\beta_{\alpha\gamma}\theta^\gamma$
where $[Z_\alpha,Z_\gamma]=C^\beta_{\alpha\gamma}Z_\beta$.  That is
to say, $\clift{Z}_\alpha(v^\beta)=-C^\beta_{\alpha\gamma}v^\gamma$.
It follows that $[\clift{Z}_\alpha,v^\beta\clift{Z}_\beta]=0$.

The second-order differential equation field $\Gamma$ may be
written
\begin{eqnarray*}
\Gamma&=& {v}^i \clift{X_i} + {v}^a \clift{\tilde{E}_a}
+ D^i \vlift{X_i} + D^a \vlift{\tilde{E}_a}\\
&=& {v}^i \clift{X_i} + {w}^a \clift{\hat{E}_a}
+ D^i \vlift{X_i} + F^a \vlift{\hat{E}_a};
\end{eqnarray*}
we have
\[
D^a=A^a_bF^b+\dot{A}^a_bw^b.
\]
By assumption $\Gamma$ is $G$-invariant, which is to say that
$[\clift{\tilde{E}_a},\Gamma]=0$.  Now
$[\clift{\tilde{E}_a},v^i\clift{X_i}]=0$, and
$[\clift{\tilde{E}_a},v^b\clift{\tilde{E}_b}]=0$, as follows from
equation~(\ref{vacoords}) and the argument that establishes it.
Moreover $[\clift{\tilde{E}_a},w^b\clift{\hat{E}_b}]=0$ since
$\clift{\tilde{E}_a}(w^b)=0$ and
$[\clift{\tilde{E}_a},\clift{\hat{E}_b}]=0$.  Thus ${v}^a
\clift{\tilde{E}_a}$ and ${w}^a \clift{\hat{E}_a}$ are both
$G$-invariant; they are not however equal, but differ by the vertical
vector field $w^b\dot{A}^a_b\vlift{\tilde{E}_a}$, which accordingly is
$G$-invariant.  Next, $[\clift{\tilde{E}_a},D^i\vlift{X_i}]=
\clift{\tilde{E}_a}(D^i)\vlift{X_i}$ since
$[\clift{\tilde{E}_a},\vlift{X_i}]=\vlift{[\tilde{E}_a,X_i]}=0$.  On
the other hand,
\begin{eqnarray*}
[\clift{\tilde{E}_a},D^b\vlift{\tilde{E}_b}]&=&
\clift{\tilde{E}_a}(D^b)\vlift{\tilde{E}_b}+
D^b[\clift{\tilde{E}_a},\vlift{\tilde{E}_b}]\\
&=&\clift{\tilde{E}_a}(D^b)\vlift{\tilde{E}_b}+
D^b\vlift{[\tilde{E}_a,\tilde{E}_b]}\\
&=&\left(\clift{\tilde{E}_a}(D^b)+
D^cC^b_{ac}\right)\vlift{\tilde{E}_b}.
\end{eqnarray*}
Finally,
$[\clift{\tilde{E}_a},F^b\vlift{\hat{E}_b}]
=\clift{\tilde{E}_a}(F^b)\vlift{\hat{E}_b}$.
The remaining coefficients of $\Gamma$ must therefore satisfy
\[
\clift{\tilde{E}_a}(D^i)=0,\quad
\clift{\tilde{E}_a}(D^b)+ D^cC^b_{ac}=0,\quad
\clift{\tilde{E}_a}(F^b)=0;
\]
that is to say, $D^i$ and $F^b$ are $G$-invariant, while the $D^a$ may be
regarded as the components of a section of $\ov{\tau}^*\ov{\la}\to
TM/G$.

Observe that
\[
\vomega(\Gamma)=v^a\clift{\tilde{E}_a}\in\vectorfields{TM}.
\]
The corresponding $\la$-valued function
$\vlift{\varpi}(\Gamma)$ is given by $\vlift{\varpi}(\Gamma)= v^aE_a$;
it is independent of the choice of $\Gamma$, as we remarked before,
and as we showed above it in fact determines a section of
$\ov{\tau}^*\ov{\la}$.

We may also express $\Gamma$ in terms of the mixed
basis:
\[
\Gamma={v}^i \clift{X_i} + {v}^a \clift{\tilde{E}_a} + D^i
\vlift{X_i} + G^a \vlift{\hat{E}_a},
\]
where
\[G^a=\bar{A}^a_bD^b=F^a+\bar{A}^a_c\dot{A}^c_bw^b,
\]
the $\bar{A}^a_b$ being the components of the inverse of the matrix
$(A^a_b)$. It is easy to see that $\clift{\tilde{E}_a}(G^b)=0$. We have
\[
\Gamma_1= {v}^i \ov{\clift{X_i}}
+ D^i \ov{\vlift{X_i}} + G^a \ov{\vlift{E_a}}.
\]
Then $\omega'(\Gamma_1)=G^a\ov{\vlift{E_a}}\in\Sec((\pTM)^*T(TM/G))$, and so
\[
\vgamma(\omega'(\Gamma_1))=G^a\vlift{\hat{E}_a}=
D^a\vlift{\tilde{E}_a} \in\vectorfields{TM}.
\]
The three-way decomposition of $\Gamma$ at the level of the $M$ square
diagram is therefore given by
\begin{eqnarray*}
\Gamma&=&\cgamma(\Gamma_2)+\vgamma(\omega'(\Gamma_1))+\vomega(\Gamma)\\
&=&({v}^i \clift{X_i}+D^i \vlift{X_i})+D^a \vlift{\tilde{E}_a}
+{v}^a \clift{\tilde{E}_a}.
\end{eqnarray*}
Among the equations for the integral curves of $\Gamma$ we find
\[
\left\{
\begin{array}{rcl}
\ddot{x}^i &=& D^i,\\
\dot{w}^a &=&F^a.
\end{array}
\right.
\]
We can write the latter as
\[
\dot{w}^a+\bar{A}^a_c\dot{A}^c_bw^b=G^a.
\]
Since the $w^a$ are $G$-invariant they can be taken, together with
$x^i$ and $v^i$, as coordinates on $TM/G$. The integral curves of
$\Gamma_1$ are the solutions of the equations
\[
\left\{
\begin{array}{rcl}
\ddot{x}^i &=& D^i,\\
\dot{w}^a+\bar{A}^a_c\dot{A}^c_bw^b&=&G^a.
\end{array}
\right.
\]
This latter equation  has a familiar structure:\ one could think of
the term $\bar{A}^a_c\dot{A}^c_b$ as representing the `angular
velocity' of the body-fixed frame with respect to the moving frame,
and $w^a$ as components of some velocity with respect to the
body-fixed frame; the whole term $\bar{A}^a_c\dot{A}^c_bw^b$ is then
of Coriolis type.

We can also write the same equation as
\[
A^a_b\dot{w}^b+\dot{A}^a_bw^b=A^a_bG^b,
\]
which is equivalent to $\dot{v}^a=D^a$.

Finally, if we rewrite ${\dot A}^c_b$ as ${\dot x}^j X_j (A^c_b) +
v^d {\tilde E}_d(A^c_b)$ and use the formulae
$X_j(A_b^c)=\gamma_j^dC^e_{db}A^c_e$ obtained earlier
(equation~(\ref{XofA})) and $\tilde{E}_d(A_b^c)=-C^c_{de}A_b^e$
(equation~(\ref{eqnforA})) we find that
\begin{eqnarray*}
{\bar A}^a_c{\dot A}^c_b w^b&=&
\bar{A}^a_c\left(\dot{x}^j\gamma_j^dC^e_{db}A^c_e-v^dC^c_{de}A_b^e\right)w^b\\
&=&\dot{x}^j\gamma_j^dC^a_{db}w^b-\bar{A}^a_cv^dv^eC^c_{de}
=\dot{x}^j\gamma_j^dC^a_{db}w^b.
\end{eqnarray*}
But $\gamma^d_jC^a_{db}=\adjconn ajb$ are the connection coefficients
of the adjoint connection; the equation for $w^a$ is therefore
equivalent to
\[
\dot{w}^a+ \adjconn aib {\dot x}^i w^b=G^a,
\]
which is in agreement with the second of equations~(\ref{decGamma1})
in the theorem.

\section{An example}

In this final section we determine the reduced
equations for an interesting class of \sode s.

The case to be discussed is that in which there is a `kinetic energy' metric
$k$ on $M$, with Levi-Civita covariant derivative $\nabla$, and the
equations of motion of the original dynamical system take the form
\[
\nabla_{\dot{c}}\dot{c}=F(c,\dot{c})
\]
for the curve $t\mapsto c(t)$ on $M$.  Such a system may be called a
system of mechanical type, with $F$ representing a force field.  We
hasten to point out, however, that according to the philosophy of
the paper as we described it in the Introduction these features of
the system are incidental to our main purpose, which is to
illustrate the methods described above using a familiar example,
rather than to discover properties of systems of mechanical type
that they have because they are systems of mechanical type.

There is great potential for confusion here, since we will now have
two connections of fundamental importance to deal with, the
Levi-Civita connection and the connection on the principal bundle
$\pM:M\to M/G$ (when we have defined the group $G$ and its action);
we warn the reader to be on guard.

We may write $\Gamma$ in the form $\Gamma=\Gamma_0+\Phi$ where
$\Gamma_0$ is the geodesic spray of the Levi-Civita connection and
$\Phi$ is the force term on the right-hand side of the equations of
motion considered as a vertical vector field on $TM$.  We now examine
the possible symmetry conditions.  For any vector field $Z$ on $M$ and
any affine spray $\Gamma_0$, $[\clift{Z},\Gamma_0]$ is vertical and
quadratic in the fibre coordinates.  On the other hand,
$[\clift{Z},\Phi]$ is vertical since $\Phi$ is; but in cases of
interest (for example, when $F$ is independent of velocities, or
linear in them, or a combination of the two) there will be no terms
quadratic in the fibre coordinates; so it is natural to consider the
situation where $[\clift{Z},\Gamma_0]$ and $[\clift{Z},\Phi]$ vanish
separately.  Now $[\clift{Z},\Gamma_0]$ vanishes if and only if $Z$ is
an infinitesimal affine transformation of the symmetric covariant
derivative defined by $\Gamma_0$, which in the case under discussion
is the Levi-Civita connection of $k$.  Since any infinitesimal
isometry is affine, it is natural to assume further that $G$ is a
group of isometries of $k$, whose elements in addition leave invariant
the force term, as represented by the vertical vector field $\Phi$.
Such a group is always a symmetry group of $\Gamma$, and in many cases the
maximal symmetry group will be of this form.

We now turn to the choice of a vector field basis on $M$ adapted to
the group action.  In this case there is a natural choice for the
connection on $\pM$:\ take its horizontal subspaces to be the
orthogonal complements of the tangent planes to the group orbits;
they are $G$-invariant since the group consists of isometries.  The
vertical vector fields $\tilde{E}_a$ comprise a basis for the
Killing fields or infinitesimal isometries.  We shall however work
with an invariant, body-fixed basis for the vertical vector fields;
that is, we choose a local basis of vector fields of the form
$\{X_i,\hat{E}_a\}$.  The components of $k$ in this basis are
denoted by $k_{ab}$, $k_{ai}$, $k_{ij}$ in the obvious fashion. The
$k_{ab}$ are evidently $G$-invariant.  By construction, $k_{ai}=0$.
The $k_{ij}$ are also $G$-invariant, and so define functions
$\ov{k}_{ij}$ on $M/G$ which are the components of the reduced
metric, say $\ov{k}$, with respect to the local vector field basis
there.  We may without loss of generality take this basis to consist
of coordinate fields, as before; the $X_i$ will not in general
commute, but $[X_i,X_j]$ will have components tangent to the group
orbits; we set $[X_i,X_j]=K_{ij}^a\hat{E}_a$ (this in effect defines
$K$ as the curvature of the connection).

The connection on $\pM$ has now been entirely taken care of;
references to a connection henceforth always mean the Levi-Civita
connection.

We set $\Phi=\Phi^i\vlift{X_i}+\Phi^a\vlift{\hat{E}_a}$; by assumption
both $\Phi^i$ and $\Phi^a$ are $G$-invariant.

In order to find the reduced system it is necessary to express
$\Gamma$ in terms of the adapted basis.  For this purpose we need the
Christoffel symbols of the Levi-Civita connection with respect to the
basis $\{X_i,\hat{E}_a\}$:\ we set
\[
\nabla_{\hat{E}_a}\hat{E}_b=\conn cab\hat{E}_c+\conn iab X_i
\]
and so on.  The order of indices is important; though the
Levi-Civita connection is symmetric, it is represented here with
respect to a non-coordinate frame. To calculate the Christoffel
symbols we need the brackets of the basis vector fields.  Recall
from equation~(\ref{hatEXbrac}) that
$[\hat{E}_a,X_i]=\gamma_i^bC^c_{ab}\hat{E}_c$, and that the
connection coefficients of the adjoint connection are given by
$\adjconn bia=\gamma_i^cC^b_{ca}$.  We therefore have the following
bracket relations:
\[
[\hat{E}_a,\hat{E}_b]=-C^c_{ab}\hat{E}_c;\quad
[X_i,\hat{E}_a]=\adjconn bia\hat{E}_b;\quad
[X_i,X_j]=K^a_{ij}\hat{E}_a.
\]
Since all of the vector fields appearing are $G$-invariant, so are
all of the coefficients on the right-hand sides. Furthermore, since
all of the brackets are vertical the Christoffel symbols with upper
index $i$ will be symmetric in their lower indices.

Using these data in the standard Koszul formulae for the Levi-Civita connection
coefficients of $k$ with respect to the basis $\{\hat{E}_a,X_i\}$ we
find that
\begin{eqnarray*}
\conn abc&=&\onehalf\left(-C^a_{bc}
+k^{ad}(k_{be}C^e_{dc}+k_{ce}C^e_{bd})\right)\\
\conn ibc&=&\onehalf k^{ij}\left(-X_j(k_{bc})
+k_{bd}\adjconn djc+k_{cd}\adjconn djb\right)\\
\conn ajb&=&\onehalf k^{ac}\left(X_j(k_{bc})-k_{bd}\adjconn djc
+k_{cd}\adjconn djb\right)\\
\conn abj&=&\onehalf k^{ac}\left(X_j(k_{bc})-k_{bd}\adjconn djc
-k_{cd}\adjconn djb\right)\\
\conn ijb&=&-\onehalf k^{ik}k_{bc}K^c_{jk}=\conn ibj\\
\conn ajk&=&\onehalf K^a_{jk}\\
\conn ijk&=&\barconn ijk,
\end{eqnarray*}
where in the final line the $\barconn ijk$ are the Christoffel
symbols of the reduced metric $\ov{k}_{ij}$.

It follows that
\begin{eqnarray*}
\Gamma_0&=&\dot{x}^i\clift{X}_i+w^a\clift{\hat{E}_a}\\
&&\mbox{}-
\left(\dot{x}^j\dot{x}^k\conn ijk+\dot{x}^jw^b(\conn ijb+\conn ibj)
+w^bw^c\conn ibc\right)\vlift{X_i}\\
&&\mbox{}-
\left(\dot{x}^j\dot{x}^k\conn ajk+\dot{x}^jw^b(\conn ajb+\conn abj)
+w^bw^c\conn abc\right)\vlift{\hat{E}_a}\\
&=&\dot{x}^i\clift{X}_i+w^a\clift{\hat{E}_a}\\
&&\mbox{}-
\left(\dot{x}^j\dot{x}^k\barconn ijk-\dot{x}^jw^b k^{ik}k_{bc}K^c_{jk}
+w^bw^ck^{ij}\left(-\onehalf X_j(k_{bc})+k_{bd}\adjconn djc\right)
\right)\vlift{X_i}\\
&&\mbox{}-
\left(\dot{x}^jw^b k^{ac}\left(X_j(k_{bc})-k_{bd}\adjconn djc\right)
+w^bw^ck^{ad}k_{be}C^e_{dc}\right)\vlift{\hat{E}_a}.
\end{eqnarray*}
The reduced equations are therefore
\begin{eqnarray*}
\ddot{x}^i+\barconn ijk\dot{x}^j\dot{x}^k&=&\Phi^i+
\dot{x}^jw^b k^{ik}k_{bc}K^c_{jk}
+w^bw^ck^{ij}\left(\onehalf X_j(k_{bc})-k_{bd}\adjconn djc\right)\\
\dot{w}^a+\adjconn ajb\dot{x}^jw^b&=&\Phi^a
-\dot{x}^jw^b k^{ac}\left(X_j(k_{bc})-k_{bd}\adjconn djc
-k_{cd}\adjconn djb\right)-w^bw^ck^{ad}k_{be}C^e_{dc}.
\end{eqnarray*}
The first equation can be written
\[
\frac{\overline{D}(\ov{k}_{ij}\dot{x}^j)}{Dt}=
\Phi_i-\dot{x}^jw^bk_{bc}K^c_{ij}
+\onehalf w^bw^c\nabla^A_{\partial/\partial x^i}(k_{bc}),
\]
where $\overline{D}/Dt$ is the covariant derivative operator of the
Levi-Civita connection of $\ov{k}$, and $\Phi_i=\ov{k}_{ij}\Phi^j$.
We can write the equation for $w^a$ in either of the following two
forms:
\begin{eqnarray*}
\frac{D^Aw^a}{Dt}&=&\Phi^a
-w^b k^{ac}\frac{D^Ak_{bc}}{Dt}-w^bw^ck^{ad}k_{be}C^e_{dc}\\
\frac{D^Aw_a}{Dt}&=&\Phi_a-w_bw_ck^{cd}C^b_{ad};
\end{eqnarray*}
to obtain the second we have used $k_{ab}$ to lower indices.

When $F=0$, that is, when $\Phi^i=\Phi^a=0$, we obtain Wong's
equations \cite{Cendra,Mont}. The case in which $F\neq0$ but $G$ is
1-dimensional is discussed by Bullo and Lewis \cite{Bullo}. We shall
show that in both cases our equations subsume those of the cited
authors.

In the case discussed in \cite{Cendra,Mont}, in addition to $F=0$ it
is assumed that the vertical part of the metric comes from a
bi-invariant metric on the Lie group $G$.  This means in the first
place that $\lie{\hat{E}_c}k(\hat{E}_a,\hat{E}_b)=0$ as well as
$\lie{\tilde{E}_c}k(\hat{E}_a,\hat{E}_b)=0$, and secondly that the
$k_{ab}$ must be independent of the $x^i$.  From the first condition
we easily find that the $k_{ab}$ must satisfy
$k_{ad}C^d_{bc}+k_{bd}C^d_{ac}=0$, and therefore $k_{ac}\adjconn
cib+k_{bc}\adjconn cia=0$.  From both together we see that the
$k_{ab}$ must be constants.  Thus $\nabla^A_{\partial/\partial
x^i}(k_{bc})=0$ and $w^bw^ck_{be}C^e_{dc}=-w^bw^ck_{de}C^e_{bc}=0$,
and the reduced equations are
\begin{eqnarray*}
\frac{\overline{D}(\ov{k}_{ij}\dot{x}^j)}{Dt}&=&-\dot{x}^jw_bK^b_{ij}\\
\frac{D^Aw_a}{Dt}&=&0.
\end{eqnarray*}
These are equivalent to the equations given in \cite{Cendra,Mont}.

In the 1-dimensional case we have a single Killing field $\tilde{E}$;
this vector field is also clearly invariant, so we shall simplify the
notation by denoting it simply by $E$ (Bullo and Lewis in fact write
$X$ for this vector field).  There is but one component of $k_{ab}$,
which is $k(E,E)=|E|^2$, and $k$ with notional upper indices is just
$|E|^{-2}$.  Furthermore, $|E|^2$ is itself invariant, and may
therefore be considered as a function on $M/G$.  An arbitrary tangent
vector $V$ may be written in the form $V=vE+v^iX_i$ (so $v$ is to be
identified with the single component of $w^a$), and since the $X_i$
are orthogonal to $E$ we have
\[
v=\frac{k(V,E)}{|E|^2};
\]
Bullo and Lewis call the map $V\mapsto k(V,E)$ the momentum map and
denote it by $J_X$.  They also introduce a type $(1,1)$ tensor field
on $M/G$ which they call the gyroscopic tensor, which they denote by
$C_X$.  The gyroscopic tensor is given essentially as follows.  The
covariant differential $\nabla E$ is a type $(1,1)$ tensor field on
$M$.  Let us denote by $E^\perp$ the distribution orthogonal to $E$,
that is, the distribution spanned by the vector fields $X_i$.  Then
$\nabla E$ may be used to define an operator on $E^{\perp}$, by first
restricting its arguments to lie in this distribution, and then
perpendicularly projecting its values into it.  Now in general we have
\[
\nabla_{X_i}\hat{E}_a=\conn jia X_j+\conn bia\hat{E}_b;
\]
so we are concerned here with $\conn jia=-\onehalf
k_{ab}k^{jk}K^b_{ik}$, albeit in the 1-dimensional case. In fact if
we write $[X_i,X_j]=K_{ij}E$ the gyroscopic tensor in
component form is
\[
C^j_i=|E|^2k^{jk}K_{ik}=|E|^2\ov{k}^{jk}K_{ik}.
\]
It is clear from this that $C^j_i$ is invariant and that
$C_{ij}=\ov{k}_{ik}C^k_j$ is skew-symmetric.  Moreover in the 1-dimensional
case $\adjconn aib=0$.  The reduced equations of motion in this case
are therefore
\begin{eqnarray*}
\ddot{x}^i+\barconn ijk\dot{x}^j\dot{x}^k&=&\Phi^i
+vC^i_j\dot{x}^j
+\onehalf v^2\ov{k}^{ij}\fpd{|E|^2}{x^j}\\
\dot{v}&=&\Phi^0
-\frac{v}{|E|^2}\dot{x}^j\fpd{|E|^2}{x^j}.
\end{eqnarray*}
Here $\Phi^0$ is the $E$-component of the force.  These equations
agree with those given by Bullo and Lewis.  These authors deal mainly
with the case in which the force is derived from a potential, and the
last term on the right-hand side of the first equation is subsumed by
them into the so-called effective potential.  Bullo and Lewis actually
give two versions of the reduced equations:\ one is in terms of $v$,
and is derived above; the other is in terms of $\mu=|E|^2v$, and the
second of the reduced equations is then simply $\dot{\mu}=0$.  But
since $|E|^2$ is the single component of $k_{ab}$, $\mu$ is just the
single component of $w_a$.  Furthermore Bullo and Lewis have
$\Phi_a=0$. So the equation $\dot{\mu}=0$ is just the second reduced
equation written in terms of $w_a$.

There is a simple explicit example which nicely illustrates both of these
cases, namely the Kaluza-Klein formulation of the equations of motion
of a charged particle in a magnetic field. The Hamiltonian and
Lagrangian approaches to this topic are well-known:\ see for
example~\cite{MandR}; here we derive the equations from those given
above. For $M$ we take $\mathrm{E}^3\times\mathrm{S}$, with
coordinates $(x^i,\theta)$. Let $A_i$ be the components of a covector
field on $\mathrm{E}^3$, and define a metric $k$ on $M$, the Kaluza-Klein
metric, by
\[
k=\delta_{ij}dx^i\odot dx^j+(A_idx^i+d\theta)^2
\]
where $(\delta_{ij})$ is the Euclidean metric. The Kaluza-Klein metric
admits the Killing field $E=\partial/\partial\theta$.  The vector
fields $X_i=\partial/\partial x^i-A_i\partial/\partial\theta$ are
orthogonal to $E$ and invariant; moreover
$k_{ij}=k(X_i,X_j)=\delta_{ij}$, while $|E|=1$. Finally
\[
[X_i,X_j]=\left(\fpd{A_i}{x^j}-\fpd{A_j}{x^i}\right)\vf{\theta}.
\]
Putting these values into the reduced equations above we obtain
\[
\ddot{x}^i=v\dot{x}^j\left(\fpd{A_i}{x^j}-\fpd{A_j}{x^i}\right),
\qquad \dot{v}=0.
\]
These are the equations of motion of a particle of unit mass and
charge $v$ in a magnetic field whose vector potential is $A_i dx^i$.

\subsubsection*{Acknowledgements}
The first author is a Guest Professor at Ghent University:\ he is
grateful to the Department of Mathematical Physics and Astronomy at
Ghent for its hospitality.

The second author is currently at The University of Michigan through
a Marie Curie Fellowship within the 6th European Community Framework
Programme. He is grateful to the Department of Mathematics for its
hospitality. He also acknowledges a research grant (Krediet aan
Navorsers) from the Research Foundation - Flanders, where he is an
Honorary Postdoctoral Fellow.

%\subsubsection*{Address for correspondence}
%65 Mount Pleasant, Aspley Guise, Beds MK17~8JX, UK;
%Crampin@btinternet.com


\begin{thebibliography}{99}

\bibitem{AM} R.\ Abraham and J.E.\ Marsden,
{\em Foundations of Mechanics}, Addison Wesley 1978.

\bibitem{Bloch} A.M.\ Bloch, {\em Nonholonomic mechanics and control},
Interdisciplinary Applied Mathematics, 24, Springer 2003.

\bibitem{Bullo} F. Bullo and A. D. Lewis, Reduction, linearization, and stability of
relative equilibria for mechanical systems on Riemannian manifolds,
{\em Acta Appl.\ Math.}, {\bf 99} (2007) 53-95.

\bibitem{Cendra}
H.\ Cendra, J.E.\ Marsden and T.S.\ Ratiu, {\em Lagrangian reduction
by stages}, Memoirs of the Am.\ Math.\ Soc.,  152, AMS 2001.

\bibitem{Cortes} J.\ Cort\'es Monforte, {\em Geometric, control
and numerical aspects of nonholonomic systems}, Lecture Notes in
Mathematics, 1793, Springer  2002.

\bibitem{DMM} M.\ de Le\'on, J.C.\ Marrero and E.\ Mart\' inez,
Lagrangian submanifolds and dynamics on Lie algebroids, {\em J.\
Phys.\ A: Math.\ Gen.} {\bf 38} (2005), R241--R308.

\bibitem{Mac} K.\ Mackenzie, {\em
Lie groupoids and Lie algebroids in differential geometry}, London
Math.\ Soc.\ Lect.\ Notes Series, 124, Cambridge Univ.\ Press 1987.

\bibitem{MMR} J.E.\ Marsden, R.\ Montgomery and T.\ Ratiu, {\em Reduction, symmetry and phases
in mechanics}, Memoirs of the Am.\ Math.\ Soc.,  88, AMS 1990.

\bibitem{MandR} J.E.\ Marsden and T.\ Ratiu, {\em Introduction to
Mechanics and Symmetry\/}, Texts in Applied Mathematics, 17,
Springer 1999.

\bibitem{Marsden} J.E.\ Marsden, T.\ Ratiu, and J.\ Scheurle,
Reduction theory and the Lagrange-Routh equations, {\em J.\ Math.\
Phys.} {\bf 41} (2000), 3379--3429.

\bibitem{MW} J.E.\ Marsden and A.\ Weinstein, Reduction of
symplectic manifold with symmetry, {\em Rep.\ Math.\ Phys.} {\bf 5}
(1974), 121--130.

\bibitem{Me} T.\ Mestdag, A Lie algebroid approach to Lagrangian
systems with symmetry, Diff.\ Geom.\ and its Appl., Proc.\ Conf.\
Prague, Aug 30-Sep 3, 2005, 523--535 (2005).

\bibitem{Me2} T.\ Mestdag, Lagrangian reduction by stages for
non-holonomic systems in a Lie algebroid framework, {\em J.\ Phys.\
A: Math.\ Gen.} {\bf 38} (2005), 10157--10179.

\bibitem{Mont} R.\ Montgomery, Canonical formulations of a classical
particle in a Yang-Mills field and Wong's equations, {\em Lett.\
Math.\ Phys.} {\bf 8} (1984), 59--67.

\bibitem{Sharpe}
R.W.\ Sharpe, {\em Differential Geometry}, Graduate Texts in
Mathematics, 166, Springer 1997.

\bibitem{Vilms}
J.\ Vilms, Connections on tangent bundles, {\em J.\
Diff.\ Geom.\/} {\bf 1} (1967), 235--243.

\bibitem{TM}
K.\ Yano and S.\ Ishihara, {\em Tangent and Cotangent Bundles},
Dekker 1973.


\end{thebibliography}
\end{document}